\documentclass[12pt]{amsart}
\usepackage{amsmath,amssymb,latexsym, amsthm, amscd, mathrsfs, stmaryrd}
\usepackage[linktocpage=true]{hyperref}
\usepackage{color} 
\usepackage[all]{xy}

\newtheorem{thm}{Theorem} [section]
\theoremstyle{definition}

\newtheorem{Def}[thm]{Definition}
\newtheorem{example}[thm]{Example}
\newtheorem{rem}[thm]{Remark}
\theoremstyle{plain}
\newtheorem{prop}[thm]{Proposition}
\newtheorem{lem}[thm]{Lemma}
\newtheorem{cor}[thm]{Corollary}

\numberwithin{equation}{section}

\newcommand{\mbf}{\mathbf}
\newcommand{\mbb}{\mathbb}
\newcommand{\mrm}{\mathrm}
\newcommand{\A}{\mathcal A}
\newcommand{\B}{\mathbf B}
\newcommand{\Bs}{\dot{\mathbf{B}}(\sll_n)}
\newcommand{\Bjs}{\dot{\mathbf{B}}^{\jmath}(\sll_n)}

\newcommand{\gl}{\mathfrak{gl}}

\newcommand{\K}{\mbf K}
\newcommand{\Ki}{\mathbf{K}^{\imath}}   
\newcommand{\Kj}{\mathbf{K}^{\jmath}}   
\newcommand{\la}{\lambda}

\newcommand{\ov}{\overline}

\newcommand{\Q}{\mathbb Q}
\newcommand{\Qq}{\mathbb{Q}(v)}

\newcommand{\Sj}{\mathbf{S}^{\jmath}}
\newcommand{\sll}{\mathfrak{sl}}
\newcommand{\Sc}{{}_\A\mathbf{S}}
\newcommand{\Scq} {\mathbf{S}}
\newcommand{\Scj}{{}_\A\mathbf{S}^{\jmath}}
\newcommand{\Sci}{{}_\A\mathbf{S}^{\imath}}
\newcommand{\Scqj} {\mathbf{S}^{\jmath}}
\newcommand{\Si}{\mathbf{S}^{\imath}(\nn, d)}

\newcommand{\tdth}{\tilde{\Theta}}
\newcommand{\txi}{\tilde{\Xi}}

\newcommand{\U}{\mbf U}

\newcommand{\Ud}{\dot{\U}}
\newcommand{\Udgl}{\dot{\U}(\gl_n)}
\newcommand{\Udgla}{ {}_\A\dot{\U}(\gl_n)}
\newcommand{\Udsl}{\dot{\U}(\sll_n)}
\newcommand{\Udsla}{ {}_\A\dot{\U}(\sll_n)}
\newcommand{\Ujgl}{\U^\jmath(\gl_n)}

\newcommand{\Ujdgl}{\dot{\U}^\jmath(\gl_n)}
\newcommand{\Ujdgla}{ {}_\A\dot{\U}^\jmath(\gl_n)}
\newcommand{\Ujsl}{\U^\jmath(\sll_n)}
\newcommand{\Ujdsl}{\dot{\U}^\jmath(\sll_n)}
\newcommand{\Ujdsla}{ {}_\A\dot{\U}^\jmath(\sll_n)}
\newcommand{\Uidsl}{\dot{\U}^\imath(\sll_{\nn})}
\newcommand{\Uidgl}{\dot{\U}^\imath(\gl_{\nn})}

\newcommand{\Z}{\mathbb Z}

\newcommand{\Ubdgl}{\mbf{i}\dot{\mbf{U}}(\gl_n)}
\newcommand{\Ubdsl}{\mbf{i}\dot{\mbf{U}}(\sll_n)}
\newcommand{\Ubd}{\mbf{i}\dot{\mbf{U}} }

\newcommand{\Sb}{\mbf{iS}}
\newcommand{\nn}{\mathfrak{n}}

\newcommand{\ro}{\mrm{ro}}
\newcommand{\co}{\mrm{co}}

\newcommand{\ibe}[1]{e_{{#1}}}
\newcommand{\ibff}[1]{f_{{#1}}}

\newcommand{\bt}{\mbf t}

\renewcommand{\l}{{\nn}} 
\newcommand{\I}{\mathbb{I}}

\title{Positivity vs negativity of canonical bases}
 
\author[Yiqiang Li]{Yiqiang Li}
\address{Department of Mathematics\\ University at Buffalo, SUNY  \\Buffalo, NY 14260}
\email{yiqiang@buffalo.edu}

\author[Weiqiang Wang]{Weiqiang Wang}
\address{ Department of Mathematics\\ University of Virginia\\ Charlottesville, VA 22904}
\email{ww9c@virginia.edu}


\begin{document}

\begin{abstract}
We provide examples for negativity of structure constants of the stably canonical basis of modified quantum $\mathfrak{gl}_n$  and 
an analogous basis of modified quantum coideal algebra of $\mathfrak{gl}_n$.
In contrast, we construct the canonical basis of the modified quantum coideal algebra of $\mathfrak{sl}_n$,
establish the positivity of its structure constants, the positivity with respect to a geometric bilinear form  as well as the positivity of its action on the tensor powers of the natural representation. The matrix coefficients of the transfer map on the associated Schur algebras with respect to the canonical bases are shown to be positive. 
Formulas for canonical basis of the Schur algebra of rank one are obtained. 
\end{abstract}

\maketitle

\begin{quote}
{\em Dedicated to  George Lusztig for his 70th birthday with admiration
}
\end{quote}
 
\setcounter{tocdepth}{1}
\tableofcontents


\section{Introduction}

\subsection{}

In \cite{BLM90}, Beilinson, Lusztig and MacPherson realized the quantum Schur algebra $\Scq(n,d)$ geometrically in terms of
pairs of partial flags of type $A$. Furthermore, they construct the modified quantum group $\Udgl$ via a stabilization procedure
from the family of algebras $\Scq (n,d)$ as $d$ varies. The IC construction provides a canonical basis for $\Scq (n,d)$ whose
structure constants  are positive (i.e., in $\mathbb N[v,v^{-1}]$), which in turn via stabilization
leads to a distinguished bar-invariant basis (which we shall refer to as {\em BLM} or {\em stably canonical basis}) for $\Udgl$.

Recently the constructions of \cite{BLM90} have been generalized to partial flag varieties of type $B$ and $C$ in \cite{BKLW} (also see \cite{FL14} for type $D$).
A family of  iSchur  algebras $\Sb (n,d)$ was realized geometrically together with canonical (=IC) bases whose
structure constants lie in $\mathbb N[v,v^{-1}]$. Via a stabilization procedure
these algebras give rise to a limit algebra which was shown to be isomorphic to the modified quantum coideal algebra $\Ubdgl$ of  $\gl_n$,
and which also admits a stably canonical basis. 
The appearance of the quantum coideal algebra was inspired by \cite{BW13} where a new approach to Kazhdan-Lusztig theory of type $B/C$ 
via a new theory of canonical bases arising from quantum coideal algebras was developed. 
Even though  the constructions for $n$ odd and even are quite different  with the case of even $n$ being more  challenging \cite{BW13}, 
one can  carry out the construction in the even n case by relating to the odd $n$ case via a more subtle two-step stabilization ~\cite{BKLW}. 

\subsection{}

The original motivation of this paper is to understand the positivity of the stably canonical basis of the modified quantum coideal algebra $\Ubdgl$.
To that end, we have to understand first the same positivity issue for $\Udgl$, as $\dot{\mbf U} (\mathfrak{gl}_{\lfloor \frac{n}2 \rfloor})$ is simpler and also 
it appears essentially
as a subalgebra of $\Ubdgl$ with compatible stably canonical bases. 
The canonical bases arising from quantum groups of ADE type are widely expected to enjoy all kinds of positivity (see \cite{L90, L93}),
and there is no indication in the literature that anything on $\Udgl$ (or $\gl_n$) differs substantially from its counterpart on $\Udsl$ (or $\sll_n$).

To our surprise,  the behavior of the BLM/stably canonical basis of $\Udgl$  turns out to be dramatically different,
already for $n=2$, from the  canonical basis of $\Udsl$. In particular, we provide examples that the structure constants of the stably canonical basis 
are negative, and that the stably canonical basis of $\Udgl$ fails to descent to the canonical basis of 
the finite-dimensional simple $\Udgl$-modules. 
These examples, though not difficult, are unexpected among the experts whom we have a chance to communicate with, so 
we write them down hoping to clarify some confusion or false expectation. 
The fundamental reason behind the failure of positivity of the BLM  basis and beyond
is that the stabilization process is not entirely geometric (when the involved matrices contain negative
diagonal entries). 

The structure constants of the canonical basis of $\Udsl$ are positive;
this follows easily from combining the positivity of the canonical (=IC) basis of the Schur algebras \cite{BLM90} with
a result of McGerty \cite[Proposition 7.8]{M12} (or with a stronger result of \cite{SV00},
which confirmed Lusztig's conjectures \cite[Conjectures 9.2, 9.3]{L99}). 
For the reader's convenience, we make explicit  this positivity in Proposition \ref{prop:posCB} and supply a short proof,
as it could not be explicitly found in these earlier papers.

\subsection{}
Now we focus on the modified quantum coideal algebra  $\Ubdsl$, for $n\ge 2$.  
We construct  a canonical basis for the modified quantum coideal algebra $\Ubdsl$ 
which shares  many remarkable properties of the canonical basis for $\Udsl$. 
In particular, it has positive  structure constants, and 
it is  characterized up to sign by the three properties: bar-invariance,  
integrality,  and almost orthonormality with respect to a bilinear form of geometric origin.
Moreover, it admits positivity with respect to the geometric bilinear form. 
In addition, 
this canonical basis is compatible with Lusztig's under a natural inclusion 
$\dot{\U}(\sll_{\lfloor \frac{n}2 \rfloor}) \subseteq \Ubd (\mathfrak{sl}_{n})$.
    
Our argument largely follows the line in McGerty's work \cite{M12} for $n$ odd (the case for $n$ even needs substantial new work), 
though we have avoided using the non-degeneracy of the geometric bilinear form of $\Ubdsl$, 
which was not available at the outset. Instead, the non-degeneracy of the bilinear form is replaced by arguments 
involving the stably canonical basis of $\Ubdgl$ 
from \cite{BKLW} and the non-degeneracy eventually follows from the almost orthonormality of the canonical basis 
which we establish. 

We further show that the transfer map on the iSchur algebras  
sends every canonical basis  element to a positive sum of canonical basis elements or zero. 
Some basic properties on the transfer map established in \cite{FL15} are used here.
Moreover, the matrix coefficients (with respect to canonical basis) for the action of any canonical basis element in $\Ubdsl$ on $\mathbb V^{\otimes d}$
are shown to be positive, where $\mathbb V$ is the $n$-dimensional natural representation of $\Ubdsl$. 
We remark that the transfer maps on the type $A$ Schur algebras were earlier studied in \cite{L99, L00, SV00, M12}.

As in \cite{BW13, BKLW}, the different behaviors in the cases for $n$ odd and even
force us to carry out the studies of the two cases separately in this paper. The case of odd $n$, indicated by the superscript $\jmath$, is easier and done first,
while the remaining case is indicated by the superscript $\imath$.
Let us set up some notations  used in the main text.
For $n$ odd and hence ${\nn}=n-1$ even, we shall denote  $\Ujdgl =\Ubd(\gl_n)$, $\Sj(n,d) =\Sb(n,d)$,
$\Uidgl =\Ubd(\gl_{\nn})$, and $\Si =\Sb({\nn},d)$.

There is another purely representation theoretic approach in \cite{BW16} toward the bilinear forms and canonical bases   for   
general quantum coideal algebras  including $\Ubdsl$, which nevertheless cannot address the positivity of canonical bases.
Note that the papers \cite{L99, L00, SV00, M12} are mostly concerned about the quantum Schur algebras and quantum groups of affine type $A$. 
A geometric setting for the quantum coideal algebras of affine type will be pursued elsewhere.

\subsection{}

The paper is organized as follows. 
In Section~\ref{sec:negativeA}, 
we construct examples   that a natural shift map (which is an algebra isomorphism) on $\Udgl$ does not preserve the BLM basis,
that the structure constants of BLM basis for $\Udgl$ are negative, 
and that the BLM basis of $\Udgl$ does not descend to the canonical basis of a finite-dimensional simple module.

In Section~\ref{sec:positiveA}, 
we show that the positivity of structure constants for the canonical basis of $\Udsl$ is an easy consequence of McGerty's results. 
Then we construct a positive basis for $\Udgl$ with positive structure constants by transporting the canonical basis of $\Udsl$. 
We explain several positivity results on the transfer map for Schur algebras.

In Sections~\ref{sec:coideal},  \ref{sec:coidealCB}, and \ref{n-even},  
we study the quantum coideal algebras   and the associated Schur algebras. 
In Section~\ref{sec:coideal}, we show the stably canonical basis constructed in \cite{BKLW} for the modified quantum coideal algebra $\Ujdgl$ 
for $n$ odd does
not have positive structure constants. 

In Section ~\ref{sec:coidealCB}, 
we set $n$ to be odd, and 
study the behavior of the canonical bases of the $\jmath$Schur algebras $\Scqj(n,d)$   and varying $d\gg 0$
under the transfer maps. This allows us to construct a canonical basis for the modified quantum coideal algebra $\Ujdsl$.
We show that the structure constants of the canonical basis of $\Ujdsl$ are positive.
We further show that the transfer map  
sends every canonical basis  element to a positive sum of canonical basis elements or zero.

In Section ~\ref{n-even}, we treat  $\Si$ and $\Uidsl$ for $\nn$ even, which is more subtle. 
We show that the main  results in Section ~\ref{sec:coidealCB} can be obtained in this case as well though extra technical work is required. 

In Section ~\ref{rank-one},
we present explicit formulas of the canonical basis of the rank one iSchur algebra in terms of the standard basis elements. 
Some interesting combinatorial identities which seem new are obtained along the way.

\vspace{.2cm}
\noindent {\bf Acknowledgements.} 
We are grateful to Huanchen Bao and Zhaobing Fan for related collaborations and many stimulating discussions.
We thank Olivier~Schiffmann and Ben Webster for very helpful comments. 
The second author is partially supported by NSF DMS-1405131; he thanks the Institute of Mathematics, Academia Sinica (Taipei)
and Institut Mittag-Leffler for an ideal working environment and support.

\section{Negativity of the stably canonical basis of $\Udgl$}
  \label{sec:negativeA}

In this section, we construct several examples which show that 
a natural shift map on $\Udgl$ does not preserve the BLM basis,
that the structure constants of BLM basis for $\Udgl$ are negative, 
and that the BLM basis of $\Udgl$ does not descend to the canonical basis of a finite-dimensional simple modules.

\subsection{The BLM preliminaries}
\label{BLM-pre}

We recall some basics from \cite{BLM90} (also see \cite{DDPW}).  Let $v$ be a formal parameter, and $\A =\Z[v, v^{-1}]$.
Let $\mbb F_q$ be a finite field of order $q$.  Let $\mathbb N =\{0,1, 2, \ldots\}$.
Let $\Sc(n,d)$  
(denoted by $\K_d$ in \cite{BLM90}) be the quantum Schur algebra over $\A$, which specializes at $v=\sqrt{q}$ to the convolution algebra of pairs
of $n$-step partial flags in $\mbb F_q^d$. The algebra $\Sc(n,d)$ admits a bar involution, a standard basis $[A]$, and a canonical (= IC) basis 
$\{A\}$ parameterized by
\begin{equation*}
\Theta_d =\Big\{ A=(a_{ij}) \in \mrm{Mat}_{n\times n} (\mbb N) \vert\; |A| = d \Big\},
\end{equation*}
where $|A|=\sum_{1\le i, j \le n} a_{ij}.$
Set  $\Theta :=\cup_{d\ge 0} \Theta_d.$
 
The multiplication formulas of the $\A$-algebras $\Sc(n,d)$ exhibit some remarkable stability as $d$ varies, which leads to a ``limit" $\A$-algebra $\K$. 
The bar involution on $\Sc(n,d)$ induces a bar involution on $\K$. The algebra $\K$ has a standard basis $[A]$
and a BLM (or {\em stably canonical}) basis $\{A\}$, parameterized by 
\begin{align*}
 \label{eq:txi}
\tdth =\{ A =(a_{ij}) \in \mbox{Mat}_{n\times n} (\mbb Z) \mid  &\; a_{ij} \geq 0\; (i\neq j) \}.
\end{align*}
Denote by $\epsilon_i$ the $i$-th standard basis element in $\Z^n$. 
For $1\le h \le n-1$, $a \geq 1$  and $\la \in \Z^n$, we denote by $E_{h,h+1}^{( a)} (\la)$ the matrix
whose $(h,h+1)$th entry is $a$, whose diagonal coincides with $\la - a \epsilon_{h+1}$, and all other entries are zero. 
Similarly, denote by $E^{( a)}_{h+1, h}(\la)$ the matrix
whose $(h+1,h)$th entry is $a$, whose diagonal coincides with $\la - a \epsilon_{h}$, and all other entries are zero.

Recall the $\A$-form of the modified quantum $\gl_n$, denoted by $\Udgla$, is generated by 
the idempotents $1_\la$ (for $\la \in \Z^n$) and the divided powers
$E_i^{(a)}1_\la$, $F_i^{(a)}1_\la$ (for $a\ge 1$ and $1\le i \le n-1$). 
It was shown in \cite{BLM90} that there is an $\A$-algebra isomorphism $\K \cong \Udgla$, which sends
$[E^{( a)}_{h,h+1}(\la)]$ to $E_h^{(a)} 1_\la$ and
$[E^{( a)}_{h+1, h}(\la)]$ to $F_h^{(a)} 1_\la$, for all admissible $\la$, $h$ and $a$.
%
We shall always make such an identification $\K \equiv \Udgla$ and use only $\Udgla$ in the remainder of the paper.

We denote 
$$
\Scq(n,d)   = \Q(v) \otimes_{\A} \Sc(n,d), \qquad
\Udgl = \Q(v) \otimes_{\A}   \Udgla.
$$
The algebra $\Udgl$ is a direct sum of subalgebras: 
\begin{equation} 
\label{eq:gld}
\Udgl =\bigoplus_{d\in \Z}  \Udgl\langle d \rangle,
\end{equation}
where $\Udgl\langle d \rangle$ is spanned by elements of the form $1_\la u 1_\mu$ with $|\mu| =|\la| =d$ and $u\in \Udgl$;
here as usual we denote $|\la| =\la_1 +\ldots +\la_n$, for $\la =(\la_1, \ldots, \la_n) \in \Z^n$. 

The elements $[E^{( a)}_{h,h+1}(\la)]$ for $E^{( a)}_{h,h+1}(\la) \in \Theta_d$ and $[E^{( a)}_{h+1,h}(\la)]$ for $E^{( a)}_{h+1,h}(\la) \in \Theta_d$ (for all admissible $h, a, \la$)
generate the $\A$-algebra $\Sc(n,d)$. 

Let $0_{i, j}$ be the $i\times j$ zero matrix.  Fix two positive integers $m, n$ such that $m < n$. Let $k\in \Z$. 
By using the multiplication formulas in ~\cite[4.6]{BLM90},  we note that the assignment 
\[
[A] \mapsto 
\left [
\begin{matrix}
A & 0_{m, n -m}\\
0_{n-m, m} & kI 
\end{matrix}
\right ]
\]
defines an algebra embedding 
\begin{align*}
\iota_{m, n}^{k}:  \ _{\mathcal A}  \dot{\mbf U} (\mathfrak{gl}_m) \longrightarrow
\Udgla.
\end{align*}

The following lemma, which basically follows from the definition of the BLM basis, will be used later on. 

\begin{lem}
\label{nm}
Let $m,n,k\in\Z$ with $0< m < n$. Then
$\iota_{m, n}^k ( \{A\} ) = \left \{ \begin{matrix}
A & 0_{m, n-m}\\
0_{n-m, m} & kI
\end{matrix}
\right \}$ 
for all $A \in \tilde \Theta$.
\end{lem}
\subsection{Incompatibility of BLM bases under the shift map}

Given $p\in \Z$,
it follows from the multiplication formulas \cite[4.6]{BLM90} that
there exists an algebra isomorphism (called {\em a  shift map})
\begin{align}
 \label{eq:shift}
\xi_p &: \Udgl \longrightarrow \Udgl,
 \qquad 
 \xi_p([A]) = [A+pI], 
\end{align}
for all  $A$ such that $A$ is either diagonal,   $ E_{h,h+1}(\lambda) $ or $E_{h+1,h} (\lambda)$  
for some  $1\le h \le n-1$ and  $I$ denotes the identity  matrix.
Note that $\xi_p$ commutes with the bar involution and $\xi_p$ preserves the $\A$-form $\Udgla$. Note also that $\xi_p^{-1} =\xi_{-p}$. 

Introduce the (not bar-invariant) quantum integers and quantum binomials, 
 for $m \in \mathbb{Z}$ and $b \in \mathbb N$,
\begin{equation}
 \label{qnumber}
\begin{bmatrix}
m\\
b
\end{bmatrix}
=\begin{bmatrix}
m\\
b
\end{bmatrix}_v
=\prod_{1\leq i\leq b} \frac{v^{2(m-i+1)}-1}{v^{2i}-1}, \quad \text{ and } \quad [m] = \begin{bmatrix}
m\\
1
\end{bmatrix} = \frac{v^{2m}-1}{v^{2}-1}.
\end{equation}
 
\begin{lem} 
\label{lem:n=2} 
Let $n=2$. 
If $a_{21}\geq 1$, $a_{22}\le -2$ and $p\le 0$, then 
\begin{align*}
\left \{
\begin{matrix}
p & 1 \\
a_{21} & a_{22} +p
\end{matrix}
\right \}
=
\left [
\begin{matrix}
p & 1\\
a_{21} & a_{22}+p
\end{matrix}
\right ]
- v^{a_{22}+1} [p+1]
\left [
\begin{matrix}
p+1 & 0\\
a_{21}-1 & a_{22} +p +1
\end{matrix}
\right ].
\end{align*}
\end{lem}

\begin{proof}
We denote the multiplication in $\Ud(\gl_2)$ by $*$ to avoid confusion with the usual matrix multiplication. 
We will repeatedly use the fact that $[A]$ is bar-invariant  (divided powers) for $A$ upper- or lower-triangular.

The formula [BLM90, 4.6(a)] gives us (for all $a_{11}, a_{22} \in \Z$ and $a_{21} \ge 1$)
\begin{align} 
 \label{eq:f1}
\left [
\begin{matrix}
a_{11} & 1\\
0 & a_{21} + a_{22}
\end{matrix}
\right ]
*
\left [
\begin{matrix}
a_{11} & 0 \\
a_{21} & a_{22} +1
\end{matrix}
\right ]
= 
\left [
\begin{matrix}
a_{11} & 1\\
a_{21} & a_{22}
\end{matrix}
\right ]
+
v^{a_{11} -a_{22} -1} \ \overline{[a_{11}+1]}
\left [
\begin{matrix}
a_{11}+1 & 0\\
a_{21}-1 & a_{22}+1
\end{matrix}
\right ].
\end{align}
By applying the bar map to \eqref{eq:f1} and then comparing  with \eqref{eq:f1} again, we have
\begin{align*}
\overline{
\left [
\begin{matrix}
a_{11} & 1\\
a_{21} & a_{22}
\end{matrix}
\right ]
}
=
\left [
\begin{matrix}
a_{11} & 1\\
a_{21} & a_{22}
\end{matrix}
\right ]
+
\left (
v^{a_{11} - a_{22} -1} \ \overline{ [a_{11}+1]} - v^{-a_{11} + a_{22} +1} [a_{11} +1]
\right )
\left [
\begin{matrix}
a_{11}+1 & 0 \\
a_{21} -1 & a_{22} +1
\end{matrix}
\right ].
\end{align*}
By a change of variables we obtain that (for $p \in \Z$)
\begin{align}
  \label{eq:f3}
\overline{
\left [
\begin{matrix}
p & 1\\
a_{21} & a_{22}+p
\end{matrix}
\right ]
}
=
\left [
\begin{matrix}
p & 1\\
a_{21} & a_{22}+p
\end{matrix}
\right ]
+
\left (
v^{ - a_{22} -1} \ \overline{ [p+1]} - v^{ a_{22} +1} [p +1]
\right )
\left [
\begin{matrix}
p+1 & 0 \\
a_{21} -1 & a_{22} +p+1
\end{matrix}
\right ].
\end{align}
Hence we can write
\begin{equation*}
\left \{
\begin{matrix}
p & 1 \\
a_{21} & a_{22}+p
\end{matrix}
\right \}
=
\left [
\begin{matrix}
p & 1\\
a_{21} & a_{22} +p
\end{matrix}
\right ]
+
x
\left [
\begin{matrix}
p+1 & 0\\
a_{21}-1 & a_{22}+p +1
\end{matrix}
\right ], \quad
\mbox{for some} \ x\in v^{-1} \mbb Z[v^{-1}].
\end{equation*}
It follows by this and \eqref{eq:f3}  that
$
x - \bar x = v^{ - a_{22} -1} \ \overline{ [p+1]} - v^{ a_{22} +1} [p +1].
$

Using the assumption that $a_{22}\le -2$ and  $p \le 0$, we have $v^{a_{22}+1} [p +1] \in v^{-1} \mbb Z[v^{-1}]$
and hence  
$x=-v^{a_{22}+1} [p +1]$.  
The  lemma follows.
\end{proof}

\begin{prop} 
\label{prop:shift}
The shift map $\xi_p: \Ud(\gl_n ) \rightarrow \Udgl$ (for $p \neq 0$) does not always preserve the BLM basis, for $n\ge 2$. 
More explicitly, for $n=2$, if $a_{21}\geq 1$, $a_{22}\le -2$ and $p<0$, then 
\begin{align*}
\xi_p 
\left \{ 
\begin{matrix}
0 & 1 \\
a_{21} & a_{22}
\end{matrix}
\right \} 
&=
\left \{
\begin{matrix}
p & 1\\
a_{21} & a_{22}+p
\end{matrix}
\right \}
+  \left (v^{-a_{22} -3} \ \overline{[p]}  + v^{a_{22}+3} [p] \right  ) 
\left \{
\begin{matrix}
p+1 & 0 \\
a_{21}-1 & a_{22} + p+1
\end{matrix}
\right \}.
\end{align*}
\end{prop}

\begin{proof}
We first verify the formula for $n=2$. 
By applying \eqref{eq:f1}  twice, we have
\begin{align}
\label{eq:f2}
\xi_p 
\left [
\begin{matrix}
a_{11} & 1\\
a_{21} & a_{22}
\end{matrix}
\right ]
=
\left [
\begin{matrix}
a_{11} + p & 1 \\
a_{21}  & a_{22}+p
\end{matrix}
\right ]
+ 
v^{-a_{11} - a_{22} -3 } \ \overline{[p]}
\left [
\begin{matrix}
a_{11} +p+1 & 0\\
a_{21}-1 & a_{22} +p +1
\end{matrix}
\right ].
\end{align}

The formula in Lemma~\ref{lem:n=2} specializes at $p=0$ to be
\begin{align*}
\left \{
\begin{matrix}
0 & 1 \\
a_{21} & a_{22}
\end{matrix}
\right \}
=
\left [
\begin{matrix}
0 & 1\\
a_{21} & a_{22}
\end{matrix}
\right ]
- v^{a_{22}+1}  
\left [
\begin{matrix}
1 & 0\\
a_{21}-1 & a_{22} +1
\end{matrix}
\right ].
\end{align*}
Hence, using \eqref{eq:f2} we have
\begin{align}
 \label{eq:f8}
\xi_p \left \{
\begin{matrix}
0 & 1 \\
a_{21} & a_{22}
\end{matrix}
\right \}
=
 \left [
\begin{matrix}
 p & 1 \\
a_{21}  & a_{22}+p
\end{matrix}
\right ]
+ 
( v^{- a_{22} -3 } \ \overline{[p]} - v^{a_{22}+1})
\left \{
\begin{matrix}
p+1 & 0\\
a_{21}-1 & a_{22}  +p+1
\end{matrix}
\right \},
\end{align}
which can be readily turned into the formula in the proposition by Lemma~\ref{lem:n=2}. 

 If $\xi_p$ preserved the BLM basis, then we would have $\xi_p(\{A\}) =\{A+pI\}$ by definitions, for all $A$.
Hence  the formula for $\xi_p \left \{
\begin{matrix}
0 & 1 \\
a_{21} & a_{22}
\end{matrix}
\right \}
$ (with $p<0$) together with the fact $\xi_p^{-1} =\xi_{-p}$ shows that $\xi_p$ (for $p\neq 0$) does not preserve the BLM basis.

The proposition for general $n\ge 2$ follows  from Lemma ~\ref{nm}.
\end{proof}

\begin{rem}
 \label{rem:shift}
It can be shown similarly that
\begin{equation*}
\xi_p 
\left \{
\begin{matrix}
0 & 1 \\
a_{21} & a_{22}
\end{matrix}
\right \}
\neq 
\left \{ 
\begin{matrix}
p & 1\\
a_{21} & a_{22} +p
\end{matrix}
\right \},\quad \text{ if  } a_{21}\geq 1, a_{22}\le -3 \text{ and } p>0.
\end{equation*}
Indeed precise formulas for both sides of this inequality can be obtained by \eqref{eq:f3} and \eqref{eq:f8}. 
\end{rem}

\begin{rem} 
  \label{rem:Phi1}
There exists a surjective algebra homomorphism $\Phi_d: \Udgl \rightarrow \Scq(n,d)$ 
which sends $[A]$ to $[A]$ for $A\in \Theta_d$ or to $0$ otherwise.
It was shown in \cite{Fu14} that $\Phi_d$ preserves the canonical bases, sending 
$\{A\}$ to $\{A\}$ for $A\in \Theta_d$ or to $0$ otherwise.
Making a $\gl_n$ analogy with \cite[9.3]{L99}, 
one might modify the map $\Phi_d$ to define a new algebra homomorphism $\Phi'_d: \Udgl \rightarrow \Scq(n,d)$ as follows:
for $u \in \Udgl\langle d-pn \rangle$ with $p\in\Z$, we let $\Phi'_d (u) = \Phi_d(\xi_p(u))$; 
also let $\Phi'_d|_{\Udgl\langle d' \rangle} =0$ unless $d' \equiv d \mod n$.
It follows by Proposition~\ref{prop:shift} and Remark~\ref{rem:shift} that $\Phi'_d: \Udgl \rightarrow \Scq(n,d)$ 
does not preserve the canonical bases for  general $d$ and $n$. 
\end{rem}

\subsection{Negativity of BLM structure constants}

\begin{prop}
\label{prop:negBLM}
The structure constants for the algebra $\Udgl$ with respect to the BLM basis are not always positive, for $n\ge 2$.
More explicitly, for $n=2$, we have 
\begin{align*}
\left
\{
\begin{matrix}
0 & 1 \\
 1&-3 
\end{matrix}
\right \}
*
\left
\{
\begin{matrix}
0 & 1 \\
 1&-3 
\end{matrix}
\right \}
&=
(v+v^{-1})^2
\left
\{
\begin{matrix}
-1 & 2 \\
 2 & -4
\end{matrix}
\right \}
- (2 v^{-2} + 1 + 2 v^2) 
\left
\{
\begin{matrix}
 0 & 1 \\
1  & -3
\end{matrix}
\right \}\\
&\quad  - (v^{-4} + v^{-2} + 2 + v^2 + v^4) 
\left
\{
\begin{matrix}
1 & 0 \\
0 & -2
\end{matrix}
\right \}.
\end{align*}
\end{prop}

\begin{proof}
It suffices to check the example for $n=2$ in view of Lemma ~\ref{nm}.  
We will repeatedly use the fact that $[A]$ is bar-invariant  (divided powers) for $A$ upper- or lower-triangular.

We claim the following identities hold:
\begin{align}  
\left \{
\begin{matrix}
0 & 1 \\
1 & -3
\end{matrix}
\right \}
&=
\left [
\begin{matrix}
0 & 1\\
1 & -3
\end{matrix}
\right ]
- v^{-2}
\left [
\begin{matrix}
1 & 0\\
0 & -2
\end{matrix}
\right ],         \label{D}
 \\
\left \{
\begin{matrix}
-1 & 2 \\
 2 & -4
\end{matrix}
\right \}
&=
\left [
\begin{matrix}
-1 & 2\\
2 & -4 
\end{matrix}
\right ],
\qquad\qquad
\left \{
\begin{matrix}
1 & 0\\
0 & -2
\end{matrix}
\right \}
=
\left [
\begin{matrix}
1 & 0\\
0 & -2
\end{matrix}
\right ].       \label{C}
\end{align}
Indeed, \eqref{D} follows by Lemma~\ref{lem:n=2}, and the second identity of \eqref{C} is clear.  
Moreover, by \cite[4.6(b)]{BLM90} and \eqref{D}, we have
\begin{eqnarray*}  
\left [
\begin{matrix}
-1 & 2\\
2 & -4 
\end{matrix}
\right ]
&=&
\left [
\begin{matrix}
1 & 0\\
2 & -4 
\end{matrix}
\right ]
*
\left [
\begin{matrix}
1 & 2\\
0 & -4 
\end{matrix}
\right ] 
+ (v^{-2} + 1 + v^2) 
\left [
\begin{matrix}
0 & 1\\
1 & -3 
\end{matrix}
\right ]
- ( v^{-4} + v^{-2} +1 )
\left [
\begin{matrix}
1 & 0\\
0 & -2
\end{matrix}
\right ] 
\\
&=&
\left [
\begin{matrix}
1 & 0\\
2 & -4 
\end{matrix}
\right ]
*
\left [
\begin{matrix}
1 & 2\\
0 & -4 
\end{matrix}
\right ] 
+ (v^{-2} + 1 + v^2) 
\left \{
\begin{matrix}
0 & 1\\
1 & -3 
\end{matrix}
\right \},
\end{eqnarray*}
which is bar invariant. Hence it must be a BLM basis element, whence \eqref{C}.

By \cite[4.6(a),(b)]{BLM90} (also see \eqref{eq:f1}), we have 
\begin{align} 
\left [
\begin{matrix}
0 & 1 \\
1 & -3
\end{matrix}
\right ]
& =
\left [
\begin{matrix}
0 & 1\\
0 & -2
\end{matrix}
\right ]
* 
\left [
\begin{matrix}
0 & 0\\
1 & -2
\end{matrix}
\right ]
 -v^2 
\left [
\begin{matrix}
1 & 0\\
0 & -2
\end{matrix}
\right ],         \label{E}
\\
\left [
\begin{matrix}
0 & 0\\
1 & -2
\end{matrix}
\right ]
*
\left [
\begin{matrix}
0 & 1\\
1 & -3
\end{matrix}
\right ]
&=(v+ v^{-1}) 
\left [
\begin{matrix}
-1 & 1\\
2 & -3
\end{matrix}
\right ]
- ( 1 + v^2 )
\left [
\begin{matrix}
0 & 0\\
1 & -2
\end{matrix}
\right ],         \label{F}
\\
\left [
\begin{matrix}
0 & 1\\
0 & -2
\end{matrix}
\right ]
*
\left [
\begin{matrix}
-1 & 1\\
2 & -3
\end{matrix}
\right ]
&= ( v + v^{-1} ) 
\left [
\begin{matrix}
-1 & 2\\
2 & -4
\end{matrix}
\right ],         \label{G}
\\
\left [
\begin{matrix}
0 & 1\\
0 & -2
\end{matrix}
\right ]
*
\left [
\begin{matrix}
0 & 0\\
1 & -2
\end{matrix}
\right ]
&=
\left [
\begin{matrix}
0 & 1\\
1 & -3
\end{matrix}
\right ]
+
v^2
\left [
\begin{matrix}
1 & 0\\
0 & -2
\end{matrix}
\right ].      \label{H}
\end{align}
Therefore we have 
\begin{align}
& \left
\{
\begin{matrix}
0 & 1 \\
 1&-3 
\end{matrix}
\right \}
*
\left
\{
\begin{matrix}
0 & 1 \\
 1&-3 
\end{matrix}
\right \}                        \notag   \\
& 
\quad = \left [
\begin{matrix}
0 & 1 \\
0 & -2  
\end{matrix}
\right ]
* 
\left [
\begin{matrix}
0 & 0 \\
1 & -2  
\end{matrix}
\right ]
* 
\left [
\begin{matrix}
0 & 1 \\
1 & -3  
\end{matrix}
\right ]
- v^{-2} 
\left [
\begin{matrix}
0 & 1 \\
0 & -2  
\end{matrix}
\right ]
*
\left [
\begin{matrix}
0 & 0 \\
1 & -2  
\end{matrix}
\right ]                         \notag  \\
& \qquad \qquad - (v^2 + v^{-2} ) 
\left [
\begin{matrix}
0 & 1 \\
1 & -3  
\end{matrix}
\right ]
+ v^{-2} ( v^2 + v^{-2} ) 
\left [
\begin{matrix}
1 & 0 \\
0 & -2  
\end{matrix}
\right ]                  \notag
\\
\quad 
&=(v + v^{-1})^2 
\left [
\begin{matrix}
-1 & 2 \\
2 & -4  
\end{matrix}
\right ]
-( 2v^{-2} + 1 + 2v^2) 
\left [
\begin{matrix}
0 & 1 \\
1 & -3  
\end{matrix}
\right ] + (v^{-4} - v^2 - v^4) 
\left [
\begin{matrix}
1 & 0 \\
0 & -2  
\end{matrix}
\right ],                      \label{R}
\end{align}
where the first identity above uses  \eqref{D} and \eqref{E}, while
the second identity above uses \eqref{F}, \eqref{G} and \eqref{H}.

With the help of \eqref{D} and \eqref{C}, a direct computation shows the right-hand side of the desired identity in the proposition
is also equal to \eqref{R}. The proposition is proved. 
\end{proof}

\subsection{Incompatability of BLM bases for $\dot{\bold U}$ and $L(\la)$}

Denote by $L(\la)$ the $\Udgl$-module of highest weight $\la$ with a highest weight vector  $u_\la^+$.

\begin{prop}
 \label{prop:CBmodule}
There exists a dominant integral weight $\la$ and some BLM basis element $C \in \Udgl$ (for $n\ge 2$) such that $C u_\la^+$ is 
not a canonical basis element of $L(\la)$.  More explicitly, for $n=2$, if $a_{21}\geq 1$, $a_{22}\le -2$ and $p\le 0$,
$\la = (p+a_{21}, a_{22}+p+1)$, then 
\begin{align*}
\left \{
\begin{matrix}
p & 1 \\
a_{21} & a_{22}+p
\end{matrix}
\right \} u_\la^+
=
v^{a_{22} +2p+3} \overline{[-a_{22}-2p-3]}
F^{(a_{21}-1)} u_\la^+.       
\end{align*}
\end{prop}

\begin{proof}
It suffices to verify such an example for $n=2$ by using Lemma ~\ref{nm} where $k$ is chosen such that $k \le a_{22}+p+1$.

By \cite[4.6]{BLM90}, we have
\begin{align*}
& \begin{bmatrix}
     p +1 & 0 \\
a_{21}  & a_{22}+p
\end{bmatrix}
*
\begin{bmatrix}
p +a_{21} & 1 \\
              0  & a_{22} +p
\end{bmatrix}
\\
&=
\begin{bmatrix}
  p         & 1 \\
a_{21}  & a_{22}+p
\end{bmatrix}
+ v^{a_{22} -1} \overline{[a_{22}+p+1]}
\begin{bmatrix}
         p +1 & 0 \\
a_{21}-1  & a_{22}+p+1
\end{bmatrix}.
\end{align*}
By plugging the above equation into the formula in Lemma~\ref{lem:n=2} (the assumption of which is satisfied), we obtain that 

\begin{align*}
\left \{
\begin{matrix}
p & 1 \\
a_{21} & a_{22}+p
\end{matrix}
\right \}
&=
\begin{bmatrix}
     p +1 & 0 \\
a_{21}  & a_{22}+p
\end{bmatrix}
*
\begin{bmatrix}
p +a_{21} & 1 \\
              0  & a_{22} +p
\end{bmatrix}
\\
& \quad 
+ v^{a_{22} +2p+3} \overline{[-a_{22}-2p-3]}
\begin{bmatrix}
p+1 & 0\\
a_{21}-1 & a_{22} +p +1
\end{bmatrix},
\end{align*}
where we have used the identity 
$$- v^{a_{22}+1} [p+1] - v^{a_{22} -1} \overline{[a_{22}+p+1]} = v^{a_{22} +2p+3} \overline{[-a_{22}-2p-3]}
$$ (note this is a bar-invariant quantum integer).


Consider the dominant  integral weight $\la = (p+a_{21}, a_{22}+p+1)$. We have
\begin{align*}
\left \{
\begin{matrix}
p & 1 \\
a_{21} & a_{22}+p
\end{matrix}
\right \} u_\la^+
&=
v^{a_{22} +2p+3} \overline{[-a_{22}-2p-3]}
\begin{bmatrix}
p+1 & 0\\
a_{21}-1 & a_{22} +p +1
\end{bmatrix} u_\la^+
 \\
&=
v^{a_{22} +2p+3} \overline{[-a_{22}-2p-3]}
F^{(a_{21}-1)} u_\la^+,         
\end{align*}
which is not a canonical basis element in $L(\la)$ 
if $-a_{22}-2p-3>1$.
\end{proof}

\begin{rem}
It is shown in \cite[Proposition~4.7]{Fu14} that the BLM basis descends to the canonical basis of $L(\la)$
when the dominant highest weight $\la$ is assumed to be in $\Z^n_{\ge 0}$. 
\end{rem}

\section{Positivity of canonical basis  of $\Udsl$ and a basis of $\Udgl$}
  \label{sec:positiveA}
  
In this section 
we exhibit various kinds of positivity of the canonical basis of $\Udsl$ and Schur algebras in relation to the transfer maps,
most of which were known by experts though probably in some other ways. 
We also construct a positive basis for $\Udgl$ by transporting the canonical basis of $\Udsl$ to $\Udgl$. 

\subsection{The algebras $\Udgl$ vs $\Udsl$}

We identify the weight lattice for $\gl_n$ as $\Z^n$ (regarded as the set of integral diagonal $n\times n$ matrices in $\tilde\Theta$
if we think in the setting of $\K$), and we define an equivalence $\sim$ on $\Z^n$ by letting
$\mu \sim \nu$ if and only if $\mu -\nu = k(1, \ldots, 1)$ for some $k\in\Z$. 
Denote by $\ov{\mu}$ the equivalence class of $\mu \in \Z^n$, and we identify the set of these equivalence classes $\bar{\Z}^n$
as the weight lattice of $\sll_n$.  We denote by $|\ov{\mu}| \in \Z/n\Z$ the congruence class of $|\mu|$ modulo $n$. 
For later use we also extend this definition to define an equivalence relation $\sim$  on $\tilde \Theta$: 
$
A \sim A' 
\mbox{ if and only if }
A - A' = k I \mbox{ for some}\ k\in \mbb Z.
$
We set
\begin{equation}  \label{barTh}
\overline{\Theta}^n = \tilde \Theta /\sim.
\end{equation}

As a variant of $\Udgl$, the modified quantum group $\Udsl$ admits a family of idempotents
$1_{\ov{\mu}}$, for $\ov{\mu} \in \bar{\Z}^n$. The algebra $\Udsl$ is naturally a direct sum of $n$ subalgebras:
\begin{equation}
 \label{eq:sld}
\Udsl =\bigoplus_{\bar{d} \in \Z/n\Z} \Udsl \langle \bar{d}\rangle, 
\end{equation}
where $\Udsl \langle \bar{d} \rangle$ is spanned by $1_{\ov{\mu}}  \Udsl 1_{\ov{\la}}$, where
$|\ov{\mu}| \equiv |\ov{\la}| \equiv d \mod n$. 
It follows that $\Udsla =\oplus_{\bar{d} \in \Z/n\Z}\; \Udsla \langle \bar{d}\rangle.$
We denote by $\pi_{\bar{d}}:  \Udsl \rightarrow \Udsl\langle \bar{d}\rangle$ the projection to the $\bar{d}$-th summand.

There exists a natural algebra isomorphism 
\begin{equation}
\label{eq:wp}
\wp_d: \Udgl \langle d\rangle \cong \Udsl \langle \ov{d}\rangle \qquad (\forall d\in \Z),
\end{equation}
which sends $1_{\la}$, $E_i
1_{\la}$ and  $F_i 1_{\la}$ to  $1_{\bar{\la}}$, $E_i 1_{\ov{\la}}$ and  $F_i 1_{\ov{\la}}$ respectively, for all $r$, $i$, 
and all $\la$ with $|\la| =d$.
This induces an isomorphism $\wp_\la: \Udgl 1_{\la} \cong \Udsl 1_{\ov{\la}}$, for each $\la \in \Z^n$,  
and also an isomorphism  ${}_\mu \wp_\la: 1_{\mu}  \Udgl 1_{\la} \cong 1_{\ov{\mu}}  \Udsl 1_{\ov{\la}}$, for all $\la, \mu \in \Z^n$ with $|\la| =|\mu|$. 
(These isomorphisms further induce similar isomorphisms for the corresponding $\A$-forms, which match the divided powers.)
Combining $\wp_d$ for all $d\in\Z$ gives us a homomorphism $\wp: \Udgl  \rightarrow \Udsl$.
It follows by definitions that 
\begin{equation}
 \label{eq:wpxi}
 \wp  \circ \xi_p = \wp, \qquad \text{ for all } p \in \Z.
 \end{equation}

Recall from Remark ~\ref{rem:Phi1}  the surjective algebra homomorphism  $\Phi_d:  \Udgl \to \Scq(n,d)$.
The algebra homomorphism  $\phi_d:  \Udsl \to \Scq(n,d)$ is defined as the composition
\begin{equation}
 \label{eq:phi}
\phi_d: \Udsl \stackrel{\pi_{\bar{d}}}{\longrightarrow}  
\Udsl\langle \bar{d}\rangle \stackrel{\wp_d}{\longrightarrow}  \Udgl\langle d\rangle  \stackrel{\Phi_d}{\longrightarrow}  \Scq(n,d).
\end{equation}
It follows  that $\phi_d|_{\mbf \Udsl \langle \bar{d'}\rangle} =0$ if $\bar{d'} \neq \bar{d}$, and we have a surjective homomorphism
$\phi_d:  \Udsl \langle \bar{d}\rangle \to \Scq(n,d)$. Clearly $\phi_d$ preserves the $\A$-forms.

\subsection{Positivity of canonical basis for $\Udsl$}

The canonical basis of $\Udsla$ (and hence of $\Udsl$) is defined by Lusztig \cite{L93}, and
it is further studied from a geometric viewpoint by McGerty \cite{M12}. 
The following positivity for canonical basis 
could (and probably should) have been formulated explicitly in \cite{M12}, as there is no difficulty to establish it therein.
Given an $n\times n$ matrix $A$, we shall denote  
$${}_pA = A+pI,$$ 
where $I$ is the identity matrix.

\begin{prop}
\label{prop:posCB}
The structure constants of the canonical basis for the algebra $\Udsl$  lie in $\mathbb N [v,v^{-1}]$, for $n\ge 2$.
\end{prop}

\begin{proof}
Let $\Bs =\cup_{\bar{d} \in \Z/n\Z} \Bs\langle \bar{d} \rangle$ be the canonical basis for $\Udsl$,
where $\Bs\langle \bar{d} \rangle$ is a canonical basis for $\Udsl \langle \bar{d} \rangle$.  
Let $a, b \in \Bs\langle \bar{d} \rangle$, for some $\bar{d}$. We have, for some suitable finite subset 
$\Omega \subset \Bs\langle \bar{d} \rangle$, 
\begin{equation}
\label{eq:abc}
a * b =\sum_{z \in \Omega} P_{a,b}^z \, z.
\end{equation}

It is shown  \cite{M12} that there exists  a positive integer $d$ in the congruence class $\bar{d}$ and $A, B, C_z \in \Theta_d$ such that
$\phi_{d+pn} (a) =\{{}_pA \}, \phi_{d+pn} (b)  =\{{}_pB\}, \phi_{d+pn} (z)  =\{{}_p C_z\}$, for all $p\gg 0$. 
Hence applying $\phi_{d+pn}$ to \eqref{eq:abc} we have
$$
\{{}_pA \} * \{{}_pB \}  =\sum_{z \in \Omega} P_{a,b}^z \, \{{}_p C_z \}.
$$
The structure constants for the canonical basis of the Schur algebra $\Scq(n,d+pn)$ are well known to be in $\mathbb N [v,v^{-1}]$
thanks to the intersection cohomology construction \cite{BLM90},
and hence $P_{a,b}^z \in \mathbb N [v,v^{-1}]$.

Since the algebra $\Udsl$ is a direct sum of the algebras $\Udsl \langle \bar{d} \rangle$ 
for $\bar{d} \in \Z/n\Z$, the proposition is proved. 
\end{proof}

\begin{rem}
The positivity as in Proposition~\ref{prop:posCB} was conjectured by Lusztig \cite{L93} for modified quantum group of symmetric type. 
There is a completely different proof of such a positivity in ADE type via categorification technique by Webster \cite{Web}.
The argument here also shows the positivity of the canonical basis of modified quantum affine $\sll_n$,
based again on McGerty's work.
\end{rem}

\subsection{Transfer map and positivity}

The transfer map for the $v$-Schur algebras
\[
\phi_{d+n, d} :\Sc(n,d+n) \longrightarrow\Sc(n,d),
\]
or $\phi_{d+n, d} :\Scq(n,d+n) \rightarrow\Scq(n,d)$ by a base change,
was defined geometrically by Lusztig   \cite{L00} and can also be described algebraically as follows. 
Set $\texttt{E}_{i;d} =\sum_{\la} [E_{i,i+1}(\la)]$ summed over all $E_{i,i+1}(\la) \in \Theta_d$,
$\texttt{F}_{i;d}  =\sum_{\la} [E_{i+1,i}(\la)]$ summed over all $E_{i+1,i}(\la) \in \Theta_d$, 
and $\texttt{K}_{\bold{a};d} =\sum_{\bold{b} \in\mathbb N^n, |\bold{b}|=d} v^{\bold{a}\cdot \bold{b}} 1_{\bold{b}}$. 
(Here $\bold{a}\cdot \bold{b} =\sum_i a_ib_i$ for $\bold{a} =(a_1, \ldots, a_n)$.)
Then $\Scq(n,d)$ is generated by these elements (see \cite{BLM90}), and the transfer map $\phi_{d+n,d}$ is characterized by
$$
\phi_{d+n, d} (\texttt{E}_{i;d+n}) = \texttt{E}_{i; d}, 
\quad
\phi_{d+n, d} (\texttt{F}_{i;d+n})  = \texttt{F}_{i;d}, 
\quad
\phi_{d+n, d} (\texttt{K}_{\bold{a};d+n})  = v^{|\bold{a}|} \texttt{K}_{\bold{a};d}. 
$$

Recall the homomorphism $\phi_d: \Udsl \rightarrow \Scq(n,d)$ from \eqref{eq:phi}. We have the following commutative diagram
by matching the Chevalley generators (see \cite{L99, L00}):
\begin{align} \label{CD:sl}
\begin{split}
\xymatrix{ 
&& \Udsl \ar@<0ex>[lld]^{\phi_{d +n}}  \ar@<0ex>[rrd]_{\phi_d} &&\\
\Scq(n,d+n) \ar@<0ex>[rrrr]^{\phi_{d +n, d}}  
&&& & \Scq(n,d)}
\end{split}
\end{align}

\begin{prop}
\label{RLconj}
The transfer map $\phi_{d+n, d} :\Scq(n,d+n) \longrightarrow\Scq(n,d)$ sends each canonical basis element to a sum of canonical basis elements
with (bar invariant) coefficients in $\mathbb N[v,v^{-1}]$ or zero.
\end{prop}

\begin{proof}
Recall that $\phi_{d+n, d}$ is the composition $(\xi\otimes \chi ) \Delta$, where 
$\xi$ and $\Delta$ are defined in ~\cite[2.2, 2.3]{L00}. 
The positivity of $\xi$ with respect to the canonical bases is clear from the definition (as it is just a rescaling operator by some $v$-powers depending on the weights). 
The positivity of $\Delta$ with respect to the canonical bases follows by its well-known identification with (the function version of)  
a hyperbolic localization functor and then appealing to the main theorem of Braden~\cite{Br03}. 

So it suffices to show the positivity of  the homomorphism $\chi: \Scq(n,n) \longrightarrow \Qq$.
Recall that the function $\chi$ is defined by
$\chi ([A]) = v^{-d_A} \det (A)$ where $d_A = \sum_{i \geq k, j< l} a_{ij} a_{kl}$.  
(Note that $\chi([A])=0$ unless $A$ is a permutation matrix.)
We claim that
\begin{align}
\label{chi-0}
\chi (\{ A\}) =
\begin{cases}
1, & \text{ if } A=I,
\\
0, & \text{ if } A\neq I
\end{cases}
\end{align}
(recall $I$ is the identity $n\times n$ matrix).
It suffices to show that the claim holds for all permutation matrices (which form the symmetric group $S_n$), and we
 prove this by induction on the length $\ell(w)$ for $w \in S_n$. 
Recall \cite{BLM90} that the canonical basis $\{w\}$ for $w\in S_n$ is simply the Kazhdan-Lusztig basis for $S_n$.
When $w =I$,  the claim holds trivially. Let $s_i$ be the $i$th elementary permutation matrix (corresponding to the $i$th simple reflection), for $1 \leq i \leq n -1$. 
It is straightforward to check by \cite[Lemma~3.8]{BLM90} that
$\{s_i\} =[s_i]+v^{-1}[I]$.  
Hence $\chi (\{s_i\})= v^{-1} \det{s_i} +v^{-1} \det{I} =0$.
Let $w \in S_n$ with $\ell (w)>1$. 
We can find an $s_i$ such that  $w = s_i w'$ with $\ell (w') +1 = \ell (w)$. 
By the construction of the Kazhdan-Lusztig basis ~\cite[\S2.2, p.170]{KL79}, we have  
\[
\{ s_i\} * \{ w' \} = \{ w \}  + \sum_{x: \ell(x) < \ell (w'), \ell (s_i x) < \ell (x)} \mu(x, w') \{ x\}, \quad \mu(x, w') \in \A.
\]
(Note the $x$ in the summation satisfies $x\neq I$.)  Now applying the 
algebra homomorphism $\chi$ to the above identity and using the induction hypothesis, we see that
$\chi (\{ w\})=0$. This finishes the proof of the claim and hence of the theorem. 
\end{proof}

\begin{prop}
 \label{prop:slsch}
The map $\phi_d: \Udsl \rightarrow \Scq(n,d)$  sends each canonical basis element to a sum of canonical basis elements
with (bar invariant) coefficients in $\mathbb N[v,v^{-1}]$ or zero.
\end{prop}

\begin{proof}
Let $b \in \Bs$. We can assume that $b\in \Bs\langle \bar{d} \rangle$ as otherwise we have $\phi_d(b)=0$. 
By \cite[Corollary 7.6, Proposition 7.8]{M12}, $\phi_{d+pn}(b)$ is a canonical basis element in $\Scq(n,d+pn)$, for some $p\gg 0$. 
Using the commutative diagram \eqref{CD:sl} repeatedly, we have
$$
\phi_d(b) = \phi_{d+n, d} \, \phi_{d+2n, d+n} \cdots \phi_{d+pn, d+pn-n} \big(\phi_{d+pn} (b) \big).
$$
It follows by repeatedly applying Proposition~\ref{RLconj} that 
the term on the right-hand side above is a sum of canonical basis elements in $\Scq(n,d)$
with coefficients in $\mathbb N[v,v^{-1}]$.
\end{proof}

\begin{rem}
 \label{rem:CBtoCB}
Proposition~\ref{RLconj}  is partly inspired by \cite[Remark 7.10]{M12}, and probably it can also be 
proved by a possible functor realization of the transfer map, whose existence was hinted at {\em loc. cit.}
Note that stronger versions of Propositions~\ref{RLconj} and \ref{prop:slsch} hold (which state that
the canonical bases are preserved by $\phi_{d+n, d}$ and $\phi_d$), according to the main results of \cite{SV00} (which proved
Lusztig's conjectures \cite{L99}). 
Our short yet transparent proofs of the weaker statements above might be of interest to the reader, and 
they will be adapted in later sections to the modified quantum coideal algebras and their associated Schur algebras. 
\end{rem}

Recall \cite{GL92} that the Schur-Jimbo $(\Scq(n,d), \bold{H}_{S_d})$-duality on $\mathbb V^{\otimes d}$ can be realized geometrically,
where $\mathbb V$ is $n$-dimensional and $\bold{H}_{S_d}$ is the Iwahori-Hecke algebra associated to the symmetric group $S_d$.
Denote by $\bold{B}(n^d)$ the canonical basis of $\mathbb V^{\otimes d}$.
The canonical bases on $\mathbb V^{\otimes d}$ as well as on $\Scq(n,d)$ are realized as simple perverse sheaves,
and the action of $\Scq(n,d)$ on $\mathbb V^{\otimes d}$ is realized in terms of a convolution product. 
Hence we have the following positivity. 

\begin{prop} \cite{GL92}
 \label{prop:sch}
The action of  $\Scq(n,d)$ on $\mathbb V^{\otimes d}$ with respect to the corresponding canonical bases is positive in the following sense:
for any canonical basis element $a$ of $\Scq(n,d)$ and any $b \in \bold{B}(n^d)$, we have
$$
a * b = \sum_{b' \in \bold{B}(n^d)} C_{a,b}^{b'} \; b', \qquad \text{ where } C_{a,b}^{b'} \in \mathbb N[v,v^{-1}].
$$
\end{prop}

We shall take the liberty of saying some action is positive in different contexts similar to  the above proposition.
Now that $\Udsl$ acts on $\mathbb V^{\otimes d}$ naturally by composing the action of $\Scq(n,d)$ on $\mathbb V^{\otimes d}$
with the map $\phi_d: \Udsl \rightarrow \Scq(n,d)$. 
We have the following corollary of Propositions~\ref{prop:slsch} and \ref{prop:sch}.

\begin{cor}
 \label{positiveact}
The action of $\Udsl$  on  $\mathbb V^{\otimes d}$ with respect to the corresponding canonical bases is positive. 
\end{cor}

Note by \cite[27.1.7]{L93} that the $d$-th symmetric power $S^{d} \mathbb V$ 
(i.e., the simple module of highest weight being $d$ times the first fundamental weight) 
is a based submodule of $\mathbb V^{\otimes d}$
in the sense of \cite[Chap. 27]{L93}, and hence $S^{d_1} \mathbb V \otimes \cdots \otimes S^{d_s} \mathbb V$ 
is also a based submodule of $\mathbb V^{\otimes d}$, where the positive integers $d_i$ satisfy  $d_1 +\ldots+ d_s =d$.
The following is now a consequence (and also a generalization) of Corollary~\ref{positiveact}.
\begin{cor}
The action of $\Udsl$  on  $S^{d_1} \mathbb V \otimes \cdots \otimes S^{d_s} \mathbb V$  with respect to the corresponding canonical bases is positive. 
\end{cor}

\subsection{A positive basis for $\Udgl$}
\label{sec:+CB}

Note that the BLM basis of $\Udgl$ restricts to a basis of $\Udgl \langle d\rangle$, which does not have
positive structure constants in general by Proposition~\ref{prop:negBLM}. 
However, in light of the positivity in Proposition~\ref{prop:posCB}, one can transport the canonical basis on $\Udsl \langle \ov{d}\rangle$
 to $\Udgl \langle d\rangle$ via the isomorphism $\wp_d$ in \eqref{eq:wp}, which has positive structure constants. Let us denote the resulting
 {\em positive basis} (or {\em can$\oplus$nical basis}) on $\Udgl =\oplus_{d\in \Z} \Udgl \langle d\rangle$ by $\B_{\text{pos}}(\gl_n)$. 
 By definition, the basis  $\B_{\text{pos}}(\gl_n)$ is invariant under the shift maps $\xi_p$ for $p\in \Z$.
 Summarizing we have the following.
 
 \begin{prop}
 There exists a positive basis $\B_{\text{pos}}(\gl_n)$ for $\Udgla$ (and also for $\Udgl$), which is induced from the canonical basis
 for $\Udsla$. 
 \end{prop}

Recall a $2$-category $\bold{\dot{\mathcal U}} (\gl_n)$ which categorifies $\Udgl$ in \cite{MSV13} is obtained by simply relabeling 
 the objects for the Khovanov-Lauda $2$-category which categorifies $\Udsl$ in \cite{KhL10}.
 We expect that the projective indecomposable $1$-morphisms in $\bold{\dot{\mathcal U}} (\gl_n)$ categorify the 
 positive basis $\B_{\text{pos}}(\gl_n)$ (instead of the BLM basis
 which has no positivity).


\section{Modified quantum coideal algebras $\Ujdgl$ and  $\Ujdsl$, for $n$ odd}
  \label{sec:coideal}

In this section and next section, we fix 2 odd positive integers $n, D$ such that
$$n =2r+1, \qquad D=2d+1. 
$$ 
We will  almost exclusively use the notation $n$ and $d$ (instead of $r$ and $D$).
We study the canonical bases for the modified quantum coideal algebras 
$\Ujdgl$ and $\Ujdsl$ as well as the $\jmath$Schur algebras $\Scqj(n,d)$. 
We will again use the notation $\{A\}, [A], \{A\}_d$ etc for the bases of these algebras, 
as these sections are independent from the earlier ones to a large extent. 
When we occasionally need to refer to similar bases in type $A$ from earlier sections, we shall add a superscript $\bold{a}$.

In this section, we  show  that the stably canonical basis constructed in \cite{BKLW} for the modified quantum coideal algebra $\Ujdgl$ does
not have positive structure constants. We also formulate some basic connections between $\Ujdgl$ and $\Ujdsl$.

\subsection{$\jmath$Schur algebras and  quantum coideal algebra}
\label{j-Schur}

We first recall some basics from \cite{BKLW}.   

Let $\mbb F_q$ be a finite field of odd order $q$.  
Let $\Scj(n,d)$  
(denoted by $\Scq^\jmath$ in \cite{BKLW}) be the $\jmath$Schur algebra over $\A$, which specializes at $v=\sqrt{q}$ to the convolution algebra of pairs
of $n$-step partial isotropic flags in $\mbb F_q^{2d+1}$ (with respect to some fixed non-degenerate symmetric bilinear form). 
The algebra $\Scj(n,d)$ admits a bar involution, a standard basis $[A]_d$, and a canonical (= IC) basis 
$\{A\}_d$ parameterized by
\begin{equation}
\label{Xi-d}
\Xi_d =\Big\{ A =(a_{ij}) \in \Theta_{2d+1} \big \vert a_{ij} = a_{n+1 -i, n+1-j}, \forall  i, j\in [1, n]\Big \}.
\end{equation}
Set  $\Xi :=\cup_{d\ge 0} \Xi_d.$
 
The multiplication formulas of the $\A$-algebras $\Scj(n,d)$ exhibits some remarkable stability as $d$ varies, which leads to a ``limit" $\A$-algebra $\Kj$. 
The bar involution on $\Scj(n,d)$ induces a bar involution on $\Kj$ \cite[\S 4.1]{BKLW}. The algebra $\Kj$ has a standard basis $[A]$
and a  stably canonical basis $\{A\}$, parameterized by 
\begin{align}
\label{txi}
 \begin{split}
\txi =\Big\{ A =(a_{ij}) \in \mbox{Mat}_{n\times n} (\mbb Z) \mid  &\; a_{ij} \geq 0\; (i\neq j), 
\\
& a_{r+1,r+1} \in 2\mathbb Z+1, a_{ij} = a_{n+1-i, n+1-j} \;(\forall i,j) \Big\}.
\end{split}
\end{align}

Recall (cf. \cite{BW13, BKLW} and the references therein)
there is a quantum coideal algebra $\Ujgl$ which can be embedded in 
$\U(\gl_n)$,
and $(\U(\gl_n), \Ujgl)$ form a quantum symmetric pair in the sense of Letzter.
For our purpose here,  its modified version $\Ujdgl$ is more directly relevant; we recall its presentation below from  \cite[\S 4.4]{BKLW} to fix some notation. 
Let  
\[
\Z_n^\jmath =\Big \{ \mu \in \mbb Z^n | \mu_i = \mu_{n+1-i}\;  (\forall i)  \text{ and } \mu_{(n+1)/2} \ \mbox{is odd} \Big \}. 
\]
Let $E_{ij}^{\theta}$ be the $n\times n$ matrix whose $(k, l)$-entry is equal to $\delta_{k,i} \delta_{l,j} + \delta_{k, n+1-i} \delta_{l,n+1-j}$.
Given  $\lambda\in \Z_n^\jmath$, we introduce the short-hand notation
$\lambda \pm \alpha_i$ whose $j$th entry is equal to 
$\lambda_j \mp (\delta_{i, j} + \delta_{n+1-i, j}) \pm (\delta_{i+1, j} + \delta_{n-i, j})$.
Recall $n=2r+1$. 
The algebra $\Ujdgl$ is the $\Qq$-algebra generated by $1_{\lambda}$, $\ibe{i}1_{\lambda}$, 
$1_{\lambda}\ibe{i}$, $\ibff{i}1_{\lambda}$ and $1_{\lambda}\ibff{i}$, for $i = 1, \dots, r$  and ${\lambda} \in \Z_n^\jmath$, 
subject to the following relations, for $i,j = 1, \dots, (n-1)/2$ and $\la, \la' \in \Z_n^\jmath$: 
\begin{eqnarray*}
\left\{
 \begin{array}{rll}
x 1_{\lambda}  1_{\lambda'} x' &= \delta_{\la, \la'} x 1_{\lambda}  x',
 \quad \text{ for } x, x'\in \{1, \ibe{i}, \ibe{j}, \ibff{i}, \ibff{j}\},   \\
\ibe{i} 1_{\lambda} &= 1_{\lambda -\alpha_i}   \ibe{i},  \\ 
\ibff{i} 1_{\lambda} &=  1_{\lambda +\alpha_i}  \ibff{i},   \\
 \ibe{i}1_{\lambda}\ibff{j} &=  \ibff{j}  1_{\lambda-\alpha_i -\alpha_j}   \ibe{i},  \quad  \text{ if } i \neq j,  \\
 \ibe{i} 1_{\lambda} \ibff{i} &=\ibff{i} 1_{\lambda -2\alpha_i}   \ibe{i}  + 
 \frac{v^{\lambda_{i+1} - \lambda_{i}} -v^{\lambda_{i} - \lambda_{i+1}}}{v-v^{-1}} 1_{\lambda-\alpha_i},     \text{ if } i \neq  \frac{n-1}2,
 \\
(\ibe{i}^2 \ibe{j} +\ibe{j} \ibe{i}^2) 1_{\lambda} &= 
 (v+v^{-1})  \ibe{i} \ibe{j} \ibe{i} 1_{\lambda},  \quad  \text{ if }  |i-j|=1, \\
( \ibff{i}^2 \ibff{j} +\ibff{j}\ibff{i}^2) 1_{\lambda} &= 
  (v+v^{-1})  \ibff{i} \ibff{j} \ibff{i}1_{\lambda} , \quad \text{ if } |i-j|=1,\\
 \ibe{i} \ibe{j} 1_{\lambda} &= \ibe{j} \ibe{i} 1_{\lambda} ,  \quad  \text{ if } |i-j|>1, \\
\ibff{i} \ibff{j} 1_{\lambda}  &=\ibff{j} \ibff{i} 1_{\lambda} ,  \quad  \text{ if } |i-j|>1, \\
 (\ibff{r}^2 \ibe{r} - 
  (v+v^{-1})  \ibff{r}\ibe{r} \ibff{r} + \ibe{r} \ibff{r}^2) 1_{\lambda} 
 &= - 
  (v+v^{-1})  \Big(v^{\lambda_{r +1} - \lambda_{r} -2} + v^{\lambda_r - \lambda_{r +1}+2} \Big)\ibff{r} 1_{\lambda},  \\
(\ibe{r}^2 \ibff{r}  - 
 (v+v^{-1})  \ibe{r}\ibff{r}\ibe{r} +\ibff{r}\ibe{r}^2) 1_{\lambda}
   &= - 
    (v+v^{-1})  \Big(v^{\lambda_{r+1} - \lambda_{r}+1} + v^{\lambda_r - \lambda_{r+1}-1}\Big) \ibe{r} 1_{\lambda}. 
   \end{array}
    \right.
\end{eqnarray*}

It was shown in \cite[\S 4.5]{BKLW} that there is an $\A$-algebra isomorphism $\Kj \cong \Ujdgla$, which matches the Chevalley generators.
we shall always make such an identification $\Kj \equiv \Ujdgla$ and use only $\Ujdgla$ in the remainder of the paper. 

Given $m\in \Z$ with $0\le2 m \le  n$, let $J_m$ be an $m \times m$ matrix whose $(i, j)$-th entry is $\delta_{i, m+1-j}$.  
Recalling the definition of  $\tilde \Theta$ depends on $n$ from Section~ \ref{BLM-pre}, 
we shall write $\tilde \Theta^n$ for $\tilde \Theta$ in this paragraph and allow $n$ vary, and so in particular $\tilde\Theta^m$ makes sense. 
To a matrix $A \in \tilde \Theta^m$ and $k\in \mbb Z$, we define a matrix 
\[
\tau_{m,n}^k  (A) =
\begin{pmatrix}
A & 0 & 0\\
0 & 2k I + \varepsilon  & 0\\
0 & 0 &J_m A J_m
\end{pmatrix}  
\]
where $\varepsilon$ is the $(n-2m)\times (n-2m)$ matrix whose only nonzero entry is 
the very central one, which equals $1$.
Thus, we have an embedding 
\[
\tau^k_{m,n}:  \tilde \Theta^m \longrightarrow  \tilde \Xi, \quad 
A \mapsto \tau^k_{m,n} (A).
\]
By comparing the multiplication formulas \cite[4.6]{BLM90} in 
$_{\mathcal A} \dot{\mbf U} ( \mathfrak{gl}_m)$ and those in $\Ujdgla$ \cite[(4.5)-(4.7)]{BKLW}, we have 
an algebra embedding, also denoted by $\tau^k_{m,n}$, 
\begin{align}
\label{tau-mn}
\tau^k_{m,n}: {}_{\mathcal A} \dot{\mbf U} ( \mathfrak{gl}_m) \longrightarrow \Ujdgla, \qquad  {}^{\bold a} [A]   \mapsto [\tau^k_{m,n} (A)].
\end{align}
(We recall here our convention of using the superscript $\bold a$ to denote the corresponding basis in the type $A$ setting from earlier sections.)
Note that the homomorphism $\tau^k_{m,n}$ commutes with the bar involutions on ${}_{\mathcal A} \dot{\mbf U} ( \mathfrak{gl}_m)$ and $\Ujdgla$.
The following lemma is immediate from the definitions. 

\begin{lem}
\label{lr} 
Suppose that $0\le m \le (n-1)/2$ and $k\in \mbb Z$. Then 
$\tau^k_{m,n} ( {}^{\bold a}\{A\}) = \{ \tau^k_{m,n}(A)\}$, for all $A\in \tilde \Theta^m$.
\end{lem}

We denote 
$$
\Scqj(n,d)   = \Q(v) \otimes_{\A} \Scj(n,d), \qquad
\Ujdgl = \Q(v) \otimes_{\A}   \Ujdgla.
$$

The quantum coideal algebra $\Ujsl$ can be embedded into (and hence identified with a subalgebra of) $\U(\sll_n)$; cf. \cite{BW13}.
We define an equivalence relation $\sim$ on $\Z_n^\jmath$: $\mu \sim \mu'$ if $\mu -\mu' =  m \sum_{i=1}^n\epsilon_i$  
 for some $m\in 2\mbb Z$.
Let $\bar \mu$ denote the equivalence class of $\mu$. Put 
\[
{}^\wedge \Z_n^\jmath  = \Z_n^\jmath / \sim.
\]
We define the $\Qq$-algebra $\Ujdsl$ formally in the same way as $\Ujdgl$ above except now that
the weights $\la, \la'$ run over ${}^\wedge \Z_n^\jmath$ (instead of $\Z^{\jmath}_n$). 
There exists a bar involution on $\Ujdsl$ (as well as on $\Ujdgl$) which fixes all the generators. 
The $\A$-form $\Ujdsla$ of  the $\Qq$-algebra $\Ujdsl$ (as well as the $\A$-form $\Ujdgla$ of $\Ujdgl$) is generated 
by the divided powers $e_i^{(a)}1_{\la}, f_i^{(a)}1_{\la}$ for all admissible $i, a, \la$.

For later use we define an equivalence relation $\sim$  on $\tilde \Xi$: 
$
A \sim A' 
\mbox{ if and only if }
A - A' = m I, \mbox{for some}\ m\in 2\mbb Z.
$
We set
\begin{equation}  \label{hatXi}
\widehat{\Xi} = \tilde \Xi/\sim.
\end{equation}

\subsection{Negativity of stably canonical basis for $\Ujdgl$}
\label{sec:negCBj}

For $a, b \in \Z$, let 
\begin{align*}
A = \begin{pmatrix}
a & 1 & 0 \\
0 & b &0 \\
0 & 1 & a
\end{pmatrix},
B = \begin{pmatrix}
a & 0 & 0 \\
1 & b &1 \\
0 & 0 & a
\end{pmatrix},
C=\begin{pmatrix}
a-1 & 1 & 0 \\
1 & b &1 \\
0 & 1 & a-1
\end{pmatrix},
D=\begin{pmatrix}
a & 0 & 0 \\
0 & b+2 &0 \\
0 & 0 & a
\end{pmatrix}.
\end{align*}

The following example arises from discussions with Huanchen Bao. 
\begin{prop}
 \label{prop:negCBj}
The structure constants for the stably canonical basis of $\Ujdgl$ are not always positive, for $n\ge 3$. More explicitly, for $n=3$ and 
for $a,b \in \Z$ with $a<b \leq -2$, the following identity holds in $\dot{\mbf U}^\jmath(\gl_3)$:
$$
\{B\} * \{A\}  =\{C\} +(v^{b+a} +v^{b-a})  \ov{[b+1]} \{D\}
$$
where  
$\ov{[b+1]}  \in \Z_{\le 0} [v,v^{-1}]$. 
\end{prop}

\begin{proof}
It suffices to check the identity for $n=3$, since the general case for $n\geq 4$ follows easily from Lemmas \ref{lr} and \ref{prop:negBLM}. 
By using \cite[(4.7)]{BKLW} we compute that
\begin{align}
 \label{eq:BA}
[B] * [A] = [C]
+ v^{-a}v^b \ov{[b+1]}  [D].
\end{align}

Observe that 
$$\{D\} =[D], \qquad
\{A\} =[A], \qquad
\{B\} =[B]
$$ 
since $D$ is diagonal, $[A]$ and $[B]$ are the Chevalley generators of $\dot{\mbf U}^\jmath(\gl_3)$.
Also note that $v^b \ov{[b+1]}$ is a bar-invariant quantum integer.  
Applying the bar involution to \eqref{eq:BA} and comparing with \eqref{eq:BA} again, we have
\begin{equation}
\label{eq:CC}
\ov{[C]} -[C] = (v^{-a} -v^a) v^b \ov{[b+1]} [D].
\end{equation}
 
By assumption that $a<b \leq -2$,  we have $v^{a+b}\ov{[b+1]} \in v^{-1} \Z_{<0} [v^{-1}]$, and hence from \eqref{eq:CC} we obtain that 
$$
\{C\} =[C] - v^{a+b}\ov{[b+1]} [D].
$$
Now the equation \eqref{eq:BA}  can be rewritten as 
$$
\{B\} * \{A\}  =\{C\} +(v^a +v^{-a}) v^b\ov{[b+1]} [D].  
$$
It is clear that $v^b \ov{[b+1]}  =-(v^{-b} +v^{-b-2} +\ldots +v^{b+2} +v^b) \in \Z_{\le 0} [v,v^{-1}]$ for $b\le -2$. 
This finishes the proof for $n=3$. 
\end{proof}

\subsection{Relating $\Ujdgl$ to $\Ujdsl$}

This subsection, in which we are making a transition from $\Ujdgl$ to $\Ujdsl$, is a preparation for the next section.

Recall that there is a Schur $(\Scq(n,d), \bold{H}_{S_d})$-duality on $\mbb V^{\otimes d}$, 
where $\mbb V$ is an $n$-dimensional vector space over $\Qq$.
It is shown \cite{G97, BW13} (see also \cite{BKLW}) that there is a Schur-type $(\Scqj(n, d), \bold{H}_{B_{d}})$-duality 
on $\mbb V^{\otimes d}$ 
where $\bold{H}_{B_d}$ is the Iwahori-Hecke algebra associated to the hyperoctahedral group $B_d$. 
In particular we have algebra homomorphisms
$$
\Scq(n, d) \stackrel{\cong}{\longrightarrow} \text{End}_{\bold{H}_{S_{d}}} (\mbb V^{\otimes d}),
\qquad
\Scqj(n, d) \stackrel{\cong}{\longrightarrow} \text{End}_{\bold{H}_{B_{d}}} (\mbb V^{\otimes d}). 
$$
Recall  the sign homomorphism 
\begin{equation}
\label{sign}
\chi_n: \Scq(n,n) \longrightarrow \Qq
\end{equation}
 from the proof of Proposition~\ref{RLconj} (cf.  \cite[1.8]{L00}).
We have a natural inclusion of algebras $\bold{H}_{B_{d}} \times \bold{H}_{S_{n}} \subseteq \bold{H}_{B_{d+n}}$.
The transfer map
$$
\phi_{d+n, d}^\jmath: \Scqj(n, d+n) \longrightarrow \Scqj(n, d)
$$
is defined as the composition of the homomorphisms
\begin{align}
\label{jphi}
\begin{split}
\Scqj(n, d+n) &
\stackrel{\cong}{\longrightarrow} \text{End}_{\bold{H}_{B_{d+n}}} (\mbb V^{\otimes (d+n)})
\stackrel{\Delta^{\jmath}}{\longrightarrow} \text{End}_{\bold{H}_{B_{d}} \times \bold{H}_{S_{n}}} (\mbb V^{\otimes (d+n)})
  \\
 &\stackrel{\cong}{\longrightarrow}   \text{End}_{\bold{H}_{B_{d}}} (\mbb V^{\otimes d}) \otimes  \text{End}_{\bold{H}_{S_{n}}} (\mbb V^{\otimes n})
\stackrel{1 \otimes \chi_n}{\longrightarrow}  
 \text{End}_{\bold{H}_{B_{d}}} (\mbb V^{\otimes d})  \overset{\cong}{\longrightarrow}
\Scqj(n, d). 
\end{split}
\end{align}
This transfer map  will be studied in depth from a geometric viewpoint in \cite{FL15}, where the proof of the following  lemma can be found. 

\begin{lem}
\label{Transfer-generator}
We have 
\[
\phi^{\jmath}_{d+n,d} ( [A]_{d+n}) = 
\begin{cases}
[A- 2I]_d, & \mbox{if} \ A- 2I \in \Xi_d,\\
0, & \mbox{otherwise.}
\end{cases}
\]
for all $A\in \Xi_{d+n}$ such that one of the following matrices is diagonal:  
$A$,  $A -  a E^{\theta}_{i+1, i}$ or $A- a E^{\theta}_{i, i+1} $ for some $a\in \mathbb N$ and $1\leq i \leq (n-1)/2$.
\end{lem}

Similar to the decomposition \eqref{eq:gld} for $\Udgl$,  we can decompose $\Ujdgl$ as a direct sum of subalgebras
\begin{equation*}
\Ujdgl =\bigoplus_{d\in \Z}  \Ujdgl\langle d \rangle,
\end{equation*}
where $\Ujdgl\langle d \rangle$ is spanned by elements of the form $1_\la u 1_\mu$ with $|\mu| =|\la| =2d+1$ and $u\in \Ujdgl.$
Also similar to the decomposition \eqref{eq:sld} for $\Udsl$, we can decompose $\Ujdsl$ as a direct sum of $n$ subalgebras
\begin{equation*}
\Ujdsl =\bigoplus_{\bar{d} \in \Z/n\Z} \Ujdsl \langle \bar{d}\rangle,
\end{equation*}
where $\Ujdsl \langle \bar{d} \rangle$ is spanned by $1_{\ov{\mu}}  \Ujdsl 1_{\ov{\la}}$, where
$|\ov{\mu}| \equiv |\ov{\la}| \equiv 2d+1 \mod 2n$.
Denote by $\pi_{\bar d}: \Ujdsl \rightarrow  \Ujdsl \langle \bar d \rangle$ the natural projection.
There exists a natural algebra isomorphism similar to \eqref{eq:wp}
\begin{equation}
\label{eq:wpjd}
\wp_{d,\jmath}: \Ujdgl \langle d\rangle \cong \Ujdsl \langle \ov{d}\rangle \qquad (\forall d\in \Z),
\end{equation}
which induces a homomorphism $\wp_\jmath: \Ujdgl  \rightarrow \Ujdsl$.  
In the same way as for $\Udgl$ defined in \eqref{eq:shift}, for each $p\in 2\Z$ we define a shift map 
\begin{align}
\label{xij}
\xi_p^\jmath &: \Ujdgl \longrightarrow \Ujdgl, \qquad \xi_p^\jmath ([A]) = [A+pI],
\end{align}
where either $A$, $A - E_{h, h+1}^{\theta}$ or $A-E_{h+1, h}^{\theta}$ for some $1\leq h \leq n-1$ is diagonal. 
It follows by definitions that 
\begin{equation}
 \label{eq:wpj}
 \wp_\jmath \circ \xi_p^\jmath = \wp_\jmath, \qquad \text{ for all } p \in 2\Z.
 \end{equation}

Recall a homomorphism $\Phi_d^\jmath: \Ujdgl \rightarrow \Scqj(n,d)$ 
was  defined in \cite[\S 4.6]{BKLW} (and denoted by $\phi_d$ therein) which 
sends $[A]$ to $[A]_d$ for $A \in \Xi_d$ and to zero otherwise. 
We define 
$$\phi_d^\jmath: \Ujdsl \longrightarrow \Scqj(n,d)
$$ 
to be the composition
\begin{equation}
  \label{eq:phidj}
\Ujdsl \stackrel{\pi_{\bar d}}{\longrightarrow}  \Ujdsl \langle \bar d \rangle \stackrel{\wp_{d,\jmath}^{-1}}{\longrightarrow} 
\Ujdgl \langle d \rangle \stackrel{\Phi_d^\jmath}{\longrightarrow}  \Scqj(n,d).
\end{equation}
  
We introduce another homomorphism
$$\Psi_d^\jmath: \Ujdgl \longrightarrow  \Scqj(n, d)
$$ 
to be the composition of the following homomorphisms
$$
\Ujdgl \stackrel{\wp_\jmath}{\longrightarrow} \Ujdsl \stackrel{\phi^\jmath_d}{\longrightarrow}  \Scqj(n, d).
$$ 
Note that $\Psi_d^\jmath \neq \Phi_d^\jmath$, but $\Psi_d^\jmath$ coincides with  $\Phi_d^\jmath$ when restricted to $\Ujdgl \langle d \rangle$.

\begin{prop}
We have the following commutative diagram:

\begin{align} \label{commutative}
\begin{split}
\xymatrix{ 
&&
\Ujdgl  \ar@<0ex>[d]^{\wp_\jmath} \ar@<0ex>[lldd]_{\Psi_{d+n}^\jmath}  \ar@<0ex>[rrdd]^{\Psi_d^\jmath}
&&\\
&& \Ujdsl \ar@<0ex>[lld]^{\phi_{d +n}^\jmath}  \ar@<0ex>[rrd]_{\phi_d^\jmath} &&\\
\Scqj (n, d+n)  \ar@<0ex>[rrrr]^{\phi_{d+n, d}^\jmath  }
&&&&  \Scqj(n, d) }
\end{split}
\end{align}
\end{prop} 

\begin{proof}
The commutativity of the left upper triangle and the right upper triangle is clear from definition. 
The commutativity of the bottom triangle follows from a description of the homomorphisms 
$\phi_d^\jmath$ and $\phi_{d+n, d}^\jmath$ 
in terms of matching the generators by Lemma ~\ref{Transfer-generator}.
\end{proof}

\section{Canonical basis for  modified quantum coideal algebra $\Ujdsl$, for $n$ odd}
\label{sec:coidealCB}

In this section we continue (as in Section~\ref{sec:coideal}) to let $n =2r+1$ and $D=2d+1$ be  odd positive integers.

We establish some asymptotical behavior for the  canonical bases of $\jmath$Schur algebras under
the transfer map. This is used to
define the canonical basis for $\Ujdsl$ and to show that  structure constants of 
the canonical basis of $\Ujdsl$ are positive. 
We further show that the transfer map on the $\jmath$Schur algebras  
sends every canonical basis  element to a positive sum of canonical basis elements or zero, and provide some corollaries.

\subsection{Asymptotic identification of canonical bases for $\Scqj(n,d)$}
\label{sec:Sjtransfer}

Recall a bilinear form $\langle \cdot , \cdot\rangle_d$ on $\Scq(n,d)$ is defined in \cite[\S 3.7]{BKLW} (and denoted by $(\cdot , \cdot)_D$ therein with $D=2d+1$).
The same argument as for \cite[Proposition 4.3]{M12} shows that 
\begin{equation}
  \label{form:slj}
\langle x, y \rangle_\jmath := \lim_{p\to \infty} \sum_{d=0}^{n-1} \big \langle \phi_{d+pn}^\jmath (x), \phi_{d+pn}^\jmath (y)  \big \rangle_{d+pn},
\quad \text{ for } x, y \in \Ujdsl,
\end{equation}
exists as an element in $\Qq$.
Thus we have constructed a bilinear form $\langle - , -\rangle_\jmath$ on $\Ujdsl$. 

Recall there is a  partial order $\preceq$ on $\txi$ \cite[(3.22)]{BKLW} by declaring
$A\preceq B$ if and only if 
$
\sum_{r\leq i; s\geq j} a_{rs} \leq \sum_{r\leq i; s\geq j} b_{rs}$ for all  $i<j.
$ 
For an $n\times n$ matrix $A = (a_{ij})$, let 
$$\text{ro}(A) = \Big(\sum_{j}a_{1j}, \sum_{j}a_{2j}, \dots, \sum_{j}a_{nj} \Big), \quad
\text{co}(A) = \Big(\sum_{i}a_{i1}, \sum_{i}a_{i2}, \dots, \sum_{i}a_{in} \Big).
$$ 
There is a partial order $\sqsubseteq$ on $\txi$ \cite[(3.24)]{BKLW}, which refines $\preceq$, so that 
$A' \sqsubseteq A$ 
if and only if $A' \preceq A$,  $\ro(A')=\ro(A)$ and $\co(A') =\co(A)$.
The following lemma is preparatory.

\begin{lem} \label{plarge}
Fix $A =(a_{ij}) \in \tilde \Xi$. 
Suppose that $p$ is an even integer such that $a_{ll}+ p\geq \sum_{i\neq j} a_{ij}$ for all $1\le l \le n$.
If $B \in \txi$ satisfies $B \sqsubseteq \, {}_{p}A$, then $B\in \Xi_{|{}_{p}A|}$, i.e., $b_{ii}\geq 0$ for all $1 \le i \le n$.
\end{lem}

\begin{proof}
We prove by contradiction. Suppose that $b_{i_0,i_0}<0$ for some $i_0$. 
We have
\[
\sum_{j\neq i_0} b_{i_0 j} > \ro(B)_{i_0} = \ro (_{p}A)_{i_0} \geq a_{i_0 i_0} + p \geq \sum_{i\neq j} a_{ij}.
\]
This implies that 
\begin{align*}
\sum_{r\leq i_0, s\geq i_0 +1} b_{rs} + \sum_{r\geq i_0, s\leq i_0 -1} b_{rs} 
& \ge \sum_{j\neq i_0} b_{i_0 j} 
\\
&>   \sum_{i\neq j} a_{ij} 
\quad \ge \sum_{r\leq i_0, s\geq i_0 +1} a_{rs} + \sum_{r\geq i_0, s\leq i_0 -1} a_{rs},
\end{align*}
which contradicts with the condition $B\sqsubseteq\ \! _{p} A$.
\end{proof}

\begin{prop} \label{M12-7.8}
Given $A \in \txi$ with $|A| =2d_0+1$, we have, for even integers $p \gg 0$,
$$\phi_{d, d-n}^{\jmath}  ( \{ _{p} A \}_d )= \{ _{(p-2)} A\}_{d-n},
$$
where we denote $d  =d_0+pn/2$ so that $|{}_{p} A| =2d+1.$
\end{prop}

\begin{proof}
The proof is essentially adapted from that of  \cite[Proposition 7.8]{M12} with minor modifications. Let us go over it for the sake of completeness.

Recall   the monomial basis $\{{}_d \texttt{M}_A \vert A\in \Xi_d\}$ of  $\mbf S^{\jmath}(n,d)$ from  ~\cite[(3.25)]{BKLW}, 
(which is denoted by $m_A$ therein).
By Lemma ~\ref{Transfer-generator} we have 
\begin{align*}
\begin{split}
& \phi^{\jmath}_{d, d-n}  ( \ \! _d \texttt{M}_A)  = \ \!  _{d-n} \texttt{M}_{A-2I}, \quad \forall d.
\end{split}
\end{align*}
(It is understood that ${}_{d-n} \texttt{M}_{A-2I}=0$ if $A-2I \not \in \Xi_{d-n}$.)
The proposition is equivalent to the following.

{\bf Claim ($\star$).}
{\em 
Let $A \in \tilde \Xi$. For all even integer $p\gg 0$, we have
\[
\{ _p A\}_d = \ _d\texttt{M}_{_pA} + \sum_{A' \prec A} c_{A', A, p} \ _{d} \texttt{M}_{_pA'},
\]
where $c_{A', A, p} \in \A$ is independent of $p\gg 0$.  
}

Recall \cite{BKLW} that the basis $\{ {}_d \texttt{M}_{_pA} \}$ satisfies $\overline{ _d \texttt{M}_{_pA}} =  {}_d \texttt{M}_{ _p A}$,
$_d \texttt{M}_{_p A}  \in  \ _{\mathcal A} \mbf S^{\jmath}(n, d)$, and
\begin{equation}
 \label{eq:c}
 {}_d \texttt{M}_{_pA} = \{ _{p}A\}_d + \sum_{B \prec A}  w_{_{p} A, \ \! _{p} B} \{_{p}B\}_d, \text{ for some }w_{_{p} A, _{p} B}\in \mathcal A.
\end{equation}

We shall argue similarly as for a claim in the proof of \cite[Proposition 7.8]{M12}, with $_D b_A$ used in {\em loc. cit.} replaced by $_d\texttt{M}_{_pA}$;
that is, we shall prove Claim~ ($\star$)  by induction on $A$ with respect to the partial order $\preceq$. 
When $A$ is minimal, it follows by \eqref{eq:c} that $_d\texttt{M}_{_pA} = \{_pA\}_d$ for all $p$, and hence Claim~ ($\star$) holds.

Now assume that Claim~ ($\star$) holds for all $B$ such that $B \prec A$. 
Set 
\[
\mathcal I_d = \big\{ B \in \tilde \Xi \big \vert  B \preceq A, {}_p B\in \Xi_d,  \ro (B) = \ro (A), \co (B) = \co (A)\big\}.
\]
Then for $p\gg 0$, we have by Lemma~\ref{plarge} that
\begin{itemize}
\item $\mathcal I_d =\{ B\in \tilde \Xi \vert  B \preceq A,   \ro (B) = \ro (A), \co (B) = \co (A)\}$; 
\item $\mathcal I_d$ is a finite set, and it is independent of $p\gg 0$ (recall $d=d_0+pn/2$ depends on $p$).   
\end{itemize}
For $u\in \A =\Z[v,v^{-1}]$, let $\deg (u)$ be its degree.
For $x \in \mbox{Span}_{\A} \{ \{ _pB\}_d | B \in \mathcal I_d\}$, we set
\[
n(x) = \mbox{max} \big \{ \deg \; \langle x, \{_pB\}_d \rangle_d \big \vert  B\in \mathcal I_d, B \neq A \big\},
\quad \text{ and } \quad n_p=n(_d\texttt{M}_{_pA}).
\]
Suppose that $n_p \geq 0$. 
We set
\[
\mathcal J_d = \big \{ B  \in \mathcal I_d \big \vert \deg \; \langle \ _d \texttt{M}_{_pA}, \{_pB\}_d\rangle_d  =n_p \big \}.
\]
Then we can write, for each $B\in \mathcal I_d$, 
\begin{align}   \label{eq:cBp}
\begin{split}
\big \langle _d \texttt{M}_{_p A}, \{_pB\}_d \big \rangle_d 
&= \sum_{i \leq n_p} c_{B,p,  i}  v^i \in \Z[v,v^{-1}],  
\\
\quad \text{ where  }  c_{B, p, i} & \in \mbb Z \;  (\forall i),   \text{ and } c_{B, p,  n_p} 
\begin{cases}
\neq 0, & \text{ if }  B\in \mathcal J_d, \\
= 0, & \text{ if }  B\in \mathcal I_d \backslash  \mathcal J_d.
\end{cases}
\end{split}
\end{align}

We define a new bar-invariant element in $_{\A} \mbf S^{\jmath}(n, d)$:
\begin{align*}
_d \texttt{M}_{_p A}' =
\begin{cases}
{}_d \texttt{M}_{_pA} - \sum_{B\in \mathcal J_d} c_{B, p,n_p}(v^{n_p} +v^{-n_p})  \{ _p B\}_d, & \text{ if } n_p>0,
\\
{}_d \texttt{M}_{_pA} - \sum_{B\in \mathcal J_d} c_{B, p,n_p}  \{ _p B\}_d, & \text{ if } n_p =0.
 \end{cases}
\end{align*}
We now show that $n( _d \texttt{M}_{_p A}' ) < n_p=n(_d\texttt{M}_{_pA})$.
We give the details for $n_p>0$, while the case for $n_p=0$ is entirely similar.
By the almost orthonormality of the canonical basis of $\Scqj(n,d)$ \cite{BKLW}, we have $\big\langle \{_pB\}_d, \{_pB'\}_d \big\rangle_d \in \delta_{B,B'} +v^{-1}\Z[v^{-1}]$.
For  $B\in \mathcal I_d$,  we have by \eqref{eq:cBp}  that
\begin{align*}
\begin{split}
\big\langle  _d \texttt{M}_{_p A}'  , \{_pB\}_d \big\rangle_d 
& = \big\langle  _d \texttt{M}_{_p A}  , \{_pB\}_d \big\rangle_d -
\sum_{B'\in \mathcal J_d} c_{B', p,n_p}(v^{n_p} +v^{-n_p})   \big\langle  \{_pB\}_d , \{_pB'\}_d \big \rangle_d 
 \\
&\equiv \sum_{i\leq n_p-1} c_{B, p, i} v^i - \sum_{B \neq B'\in \mathcal J_d} c_{B', p, n_p}  v^{n_p}  \big \langle  \{_pB\}_d , \{_pB'\}_d\big \rangle_d
 \quad \mbox{mod} \ v^{-1} \mbb Z[v^{-1}], 
\end{split}
\end{align*}
which implies that $n( _d \texttt{M}_{_p A}' ) < n_p$. 


By repeating the above procedure with ${}_d \texttt{M}_{_p A}'$ in place of ${}_d \texttt{M}_{_p A}$, we produce
a bar-invariant element ${}_d \texttt{M}_{_p A}''$ in $_{\A} \mbf S^{\jmath}(n, d)$ with degree $n( _d \texttt{M}_{_p A}'') <n( _d \texttt{M}_{_p A}' )$,
and then repeat again and so on.
So under the assumption that $n_p\ge 0$, after finitely many steps we obtain a bar-invariant element in $_{\A} \mbf S^{\jmath}(n, d)$, denoted by 
 $\mbf b_{_pA}$, with $n( \mbf b_{_pA}) <0$. 
 
On the other hand, if $n_p =n( _d \texttt{M}_{_p A})<0$, then we simply set  $\mbf b_{_pA} = {}_d \texttt{M}_{_p A}$.

We now show that $\mbf b_{_pA} = \{_pA\}_d$.
By the above construction and \eqref{eq:c}, we have 
\[
\mbf b_{_pA} = \{_pA\}_d + \sum_{B\in \mathcal I_d} f_B \{_pB\}_d,
\]
for some $f_B\in \A$ and $\overline{f_B} =f_B$.
If $f_B \neq 0$ for some $B$, then $n(\mbf b_{_pA})\geq 0$, which is a contradiction. Hence we have
$\mbf b_{_pA} = \{_pA\}_d$.

In the finite process above of constructing $\{_pA\}_d$ (in the form of $\mbf b_{_pA}$) from the monomial basis, we only need the first $n_p$ coefficients 
of  $\langle _d \texttt{M}_{_p A}, \{_pB\}_d \rangle_d$ as well as of $\langle \{_pB'\}_d,  \{_pB\}_d\rangle_d$ for $B\in \mathcal I_d, B'\in \mathcal J_d$. 
Recall that the monomial basis $\{M_A \vert A\in \txi\}$ of $\mbf K^{\jmath}$ from ~\cite[4.8]{BKLW}  satisfies that
$
\phi_d (M_A) = {}_d\texttt{M}_{_pA}$ if ${}_p A\in \Xi_d.$
So by  the inductive assumption that any element $B \prec A$ satisfies Claim~ ($\star$) and the convergence of the bilinear form 
$\langle\cdot, \cdot \rangle_d$ (with $d=d_0 +pn/2$) in $\Q((v^{-1}))$   as $p \mapsto \infty$, we 
conclude that  
$\mathcal I_d, n_p$ and $c_{B, p, i} \, (0 \leq i \leq n_p)$ are all independent of $p\gg 0$.
Now Claim ~($\star$) follows by the   construction of $\{_pA\}_d$ as $\mbf b_{_pA}$ in terms of the monomial basis above.
\end{proof}

 

\begin{prop}
 \label{prop:xiCB}
Given $A \in \txi$, we have
$$
\xi_{-2}^\jmath  ( \{ _{p} A \} )= \{ _{(p-2)} A\}, 
\qquad
\wp_\jmath ( \{ _{p} A \} )= \wp_\jmath (\{ _{(p-2)} A\})
$$ 
for all even integers $p \gg 0$, where $\xi_{-2}^{\jmath}$ is defined in (\ref{xij}).
\end{prop}

\begin{proof}
Denote $|A|=2d_0+1$, and $d  =d_0+pn/2.$
We have the following commutative diagram
\begin{eqnarray}
\label{CD:old}
\begin{CD}
\Ujdgl  @> \xi_{-2}^\jmath  >> \Ujdgl \\
@V\Phi_d^\jmath VV @V\Phi_{d-n}^\jmath VV \\
\Scqj(n,d) @>\phi_{d, d-n}^\jmath >>\Scqj(n,d-n)
\end{CD}
\end{eqnarray}
i.e., $\Phi_{d-n}^\jmath \circ \xi_{-2}^\jmath  = \phi_{d, d-n}^\jmath \circ \Phi_d^\jmath$.
By ~\cite[Appendix A, Theorem~A.21]{BKLW}, we have
\begin{equation}
 \label{eq:6.10}
\Phi_d^\jmath (\{_{p}A\}) = \{ _{p} A\}_d, \qquad
\Phi_{d-n}^\jmath (\{ _{(p-2)} A\}) =\{ _{(p-2)} A\}_{d-n}, \quad \forall p\gg0.
\end{equation}
Moreover, by ~\cite[(4.8)]{BKLW}, we have
\begin{equation}
 \label{eq:4.8}
\xi_{-2}^\jmath  (\{_{p}A\} ) =\{_{(p-2)}A\} + \sum_{B \in \Xi_{d-n}} f_B \{B\}, \qquad (\text{for } f_B \in \A),
\end{equation}
where the summation can be taken over $B \in \Xi_{d-n}$ is ensured by Lemma ~\ref{plarge}.

Using  Proposition ~\ref{M12-7.8}, \eqref{eq:6.10}, \eqref{CD:old}, and \eqref{eq:4.8} one by one, we conclude that
\begin{align*}
\{_{(p-2)}A\}_{d-n}   &= \phi_{d, d-n}^\jmath \circ \Phi_d^\jmath (\{_{p}A\} ) 
\\
&=\Phi_{d-n}^\jmath \circ \xi_{-2}^\jmath (\{_{p}A\} ) 
=
\{_{(p-2)}A\}_{d-n} + \sum_{\stackrel{B \in \Xi_{d-n}}{B\sqsubset\ \! _{(p-2)} A}} f_B \{B\}_{d-n}.
\end{align*}
Hence all $f_B$ must be zero, and the first identity in the proposition follows from \eqref{eq:4.8}. 
The second identity is immediate from the first one and \eqref{eq:wpj}. 
\end{proof}

\subsection{Canonical basis for $\Ujdsl$}
 \label{sec:CBcoidealj}

By Proposition~\ref{prop:xiCB}, for $\widehat A \in \widehat{\Xi}$ (recall $\widehat{\Xi}$ from \eqref{hatXi}), the element 
$$
b_{\widehat A} := \wp_\jmath ( \{ _{p} A \} ), \qquad \text{ for } p \gg 0
$$ 
is independent of $p$  and thus a well-defined element in $\Ujdsl$. 
It follows by definition that $\wp_\jmath: \Ujdgl  \rightarrow \Ujdsl$ preserves the $\A$-forms, so we have 
$b_{\widehat A}  \in \Ujdsla$. 

\begin{prop}  \label{phi-2p}
For $A\in \txi$ with $|A|=2d_0+1$, let $d =d_0 + pn/2$. Then
$\phi_d^\jmath ( b_{\widehat A} ) = \{ _{p} A\}_d$ for even integers $p\gg 0$.
\end{prop}

\begin{proof}
We have, for $p\gg0$, 
\[
\phi_d^\jmath ( b_{\widehat A} ) = 
\phi_d^\jmath ( \wp_\jmath (\{ _{p} A\}))
=\Psi_d^\jmath (\{_{p} A\})
=\Phi_d^\jmath (\{_{p}A\})
= \{_{p} A\}_d,
\]
where the first equality follows by definition, the second one is due to  (\ref{commutative}), 
the third one follows by definition \eqref{eq:phidj},  and the last one follows from ~\cite[Theorem 6.10]{BKLW}.
The proposition is proved. 
\end{proof}

\begin{thm}
\label{j-CB}
The set $\Bjs = \{ b_{\widehat A} \vert  \widehat A\in  \widehat{\Xi}  \}$ forms a basis of $\Ujdsl$, and
it also forms an $\A$-basis for $\Ujdsla$. 
\end{thm}

\begin{proof}
Observe that 
$\xi_{p} (\{ A\} ) = \{ A + p I\} + \mbox{lower terms}.$
Hence it follows by the  surjectivity of $\wp$ that $\Bjs$ is a spanning set for the $\A$-module $\Ujdsla$. To show that
$\Bjs$ is linearly independent, it suffices to check that $\Bjs \cap \Ujdsl \langle \bar d\rangle$ is linearly independent for each $\bar d \in \Z/n\Z$.  
This is then reduced to the $\jmath$Schur algebra level by Proposition ~\ref{phi-2p}, which is clear.
Hence $\Bjs = \{ b_{\widehat A} \vert  \widehat A\in  \widehat{\Xi}  \}$ is an $\A$-basis of $\Ujdsla$,
and thus it is also a basis of $\Ujdsl$.
\end{proof}

\subsection{Positivity of the canonical basis $\Bjs$}
 \label{sec:CBcoidealj+}

The basis $\Bjs$ is called the {\em canonical basis} (or {\em $\jmath$-canonical basis}) of $\Ujdsl$,
as we shall show that the canonical basis $\Bjs$ admits several remarkable properties such as positivity and almost orthonormality
just like Lusztig's canonical basis for $\Udsl$ (see Proposition~\ref{prop:posCB} and \cite{L93}). 

Given $\widehat{A}, \widehat{B} \in \widehat{\Xi}$, we write 
\[
b_{\widehat{A}} * b_{\widehat{B}} =\sum_{\widehat C\in   \widehat{\Xi}} P_{\widehat A, \widehat B}^{\widehat C} \; b_{\widehat C},
\]
where $P_{\widehat A, \widehat B}^{\widehat C} \in \Z[v,v^{-1}]$ is zero for all but finitely many $\widehat C$.

\begin{thm} [Positivity]
 \label{th:posCBj}
We have $P_{\widehat A, \widehat B}^{\widehat C} \in \mathbb N [v, v^{-1}]$, for any $\widehat A, \widehat B, \widehat C \in   \widehat{\Xi}$.
\end{thm}

\begin{proof}
Let us write $b_{\widehat{A}}* b_{\widehat{B}} =\sum_{\widehat C\in \Omega} P_{\widehat A, \widehat B}^{\widehat C} \; b_{\widehat C}$,
where $\Omega$ is the finite set which consists of $\widehat C \in \widehat{\Xi}$ such that $P_{\widehat A, \widehat B}^{\widehat C} \neq 0$. 
Let us pick representatives $A, B, C \in \txi$ 
such that $|A| = |B| = |C| =2d_0+1$ for all $\widehat C \in \Omega$.
 
By Proposition ~\ref{phi-2p}, we can find some large $p$ (and recall $d =d_0 + pn/2$) such that 
${}_{p} A, {}_p B, {}_{p} C \in \Xi$ and 
\[
\phi_d^\jmath ( b_{\widehat A} ) = \{ _{p} A\}_d,
\quad
\phi_d^\jmath ( b_{\widehat B} ) = \{ _{p} B\}_d,
\quad
\phi_d^\jmath ( b_{\widehat C} ) = \{ _{p} C\}_d,
\]
for all $C$ with $\widehat C \in \Omega$.
So we have the following multiplication of canonical basis in $\Scqj(n,d)$:
\[
\{ {}_p A\}_d * \{  {}_p  B\}_d
= \sum_{\widehat C\in \Omega} P_{\widehat A, \widehat B}^{\widehat C} \;
\{  {}_p  C\}_d.
\]
Thanks to the intersection cohomology construction of the canonical basis for $\Scqj(n, d)$ \cite{BKLW},  
the structure constants $P_{\widehat A, \widehat B}^{\widehat C}$ lie in $\mathbb N[v,v^{-1}]$. This proves the theorem.  
\end{proof}

\begin{prop}
\label{orthonormality}
The bilinear form $\langle \cdot , \cdot\rangle_\jmath$ on $\Ujdsl$ is non-degenerate.
Moreover,  the almost orthonormality for the canonical basis holds:
$\langle b_{\widehat A}, b_{\widehat B}\rangle_{\jmath} \in \delta_{\widehat A, \widehat B} + v^{-1} \mbb Z[[v^{-1}]]$. 
\end{prop}

\begin{proof}
This almost orthonormality follows by an argument entirely similar to  \cite[Theorem~ 8.1]{M12},
and it implies the non-degeneracy of  the bilinear form. 
\end{proof}


We  have the following positivity for the canonical bases with respect to the bilinear form.
\begin{thm}
 \label{thm:positiveform}
We have
$\langle b_{\widehat A}, b_{\widehat B}\rangle_\jmath  = \delta_{\widehat A, \widehat B} + v^{-1} \mathbb N [[v^{-1}]]$,
for any $\widehat A, \widehat B \in   \widehat{\Xi}$.
\end{thm}

\begin{proof}
The proof follows very closely McGerty's geometric argument ~\cite[Proposition 6.5, Theorem 8.1]{M12}, 
with ~\cite[Corollary 3.3]{M12} replaced by ~\cite[Corollary 3.15]{BKLW}.
We only sketch the proof with an emphasis on the difference  and refer to {\em loc. cit.} for further details.

By the definition of $\langle \cdot,\cdot \rangle_{\jmath}$, it is reduced to show that 
$\langle \{ A\}_d, \{B\}_d \rangle_d \in \delta_{A, B} + v^{-1} \mathbb N  [v^{-1} ]$ for all $A, B \in \Xi_d$ 
where $\langle \cdot,\cdot \rangle_d$ is the bilinear form on $\Scqj(n,d)$.
The positivity of the form $\langle \cdot,\cdot \rangle_d$ in the theorem will follow by its identification with another geometrically defined bilinear form 
$\langle \cdot,\cdot  \rangle_{g, d}$ on $\Scqj(n,d)$ which manifests the positivity.
The latter is defined exactly the same as ~\cite[(6-1)]{M12} with the flag variety $\mathscr F_{\bf a}$ therein replaced by
the $n$-step isotropic flag variety of a $(2d+1)$-dimensional  complex vector space equipped with a non-degenerate symmetric bilinear form.

Now arguing similar to  \cite[Lemma~ 6.3]{M12}, we have,  for all $A$ minimal with respect to the partial order $\leq$, 
\[
\langle \{A\}_d * \{B\}_d, \{C\}_d \rangle_{g,d} = v^{d_A - d_{A^t}} \langle \{B\}_d, \{A^t\}_d * \{C\}_d \rangle_{g,d},
\] 
where $A^t$ is the transpose of $A$. 
This implies the analog of ~\cite[Lemma 6.4]{M12}, which gives the formulas for the adjoints of the 
Chevalley generators of $\Scqj(n,d)$ for the bilinear form $\langle \cdot,\cdot  \rangle_{g, d}$, 
and we observe that they coincide with the ones for $\langle \cdot,\cdot \rangle_d$ given in \cite[Corollary 3.15]{BKLW}.
Hence, the identification of the forms $\langle \cdot,\cdot  \rangle_d$ and $\langle \cdot,\cdot  \rangle_{g,d}$ 
is reduced to show that 
\[
\langle \{A\}_d, \{ D_{\lambda}\}_d \rangle_d = \langle \{A\}_d, \{ D_{\lambda}\}_d \rangle_{g,d}, \quad \forall A, \lambda
\]
where $D_{\lambda}$ is the diagonal matrix with diagonal $\lambda$. 
Indeed, if we write $\{ A\}_d = \sum_{A'\leq A} P_{A, A'} [A']_d$ for some $P_{A, A'} \in \mbb Z[v^{-1}]$, 
then both sides of the above equation are equal to $P_{A, D_{\lambda}}$ if $\ro (A) = \co (A) = \lambda$, or zero otherwise. The theorem follows.
\end{proof}

Furthermore, we have the following characterization of the signed canonical basis. 

\begin{prop}
\label{3-property}
The  signed canonical basis $-\Bjs \cup \Bjs$ is characterized by the following three properties:
(i) $\overline{b} = b$, 
(ii) $b\in \Ujdsla$,
and (iii) $\langle b, b' \rangle_{\jmath} \in \delta_{b, b'} + v^{-1} \Z[[v^{-1}]]$. 
\end{prop}

\begin{proof}
It follows by definition and Proposition~\ref{orthonormality} that  $-\Bjs \cup \Bjs$ satisfies the three properties above.
The characterization claim is then proved in the same way as ~\cite[14.2.3]{L93} for the usual canonical bases.
\end{proof}

\subsection{Positivity of transfer map $\phi^{\jmath}_{d+n, d}$}
\label{pos-transfer-cb}
 
We have  the following positivity on the transfer map $\phi^{\jmath}_{d+n, d}$, generalizing Proposition~\ref{RLconj} on the
positivity of the transfer map $\phi_{d+n, d}$.

\begin{thm}
\label{j-RLconj}
The transfer map $\phi_{d+n, d}^\jmath :\Scqj(n,d+n) \rightarrow\Scqj(n,d)$ sends each canonical basis element to a sum of canonical basis elements
with (bar invariant) coefficients in $\mathbb N[v,v^{-1}]$.
\end{thm}

\begin{proof}
The strategy of the proof is identical to the one for  Proposition ~\ref{RLconj},
which is reduced to the positivity of $\Delta^{\jmath}$ defined in (\ref{jphi}) with respect to the canonical bases and 
the positivity of $\chi$ which was already established in \eqref{chi-0}.
The proof of the  positivity of $\Delta^{\jmath}$ is 
similar to that of $\Delta$ in the proof of Proposition~\ref{RLconj} 
(the details are provided in ~\cite{FL15} together with other applications in a geometric setting).
\end{proof}

\begin{prop}
 \label{prop:j-slsch}
The map $\phi_d^\jmath: \Ujdsl \rightarrow \Scqj(n,d)$ sends each canonical basis element to a sum of canonical basis elements
with (bar invariant) coefficients in $\mathbb N[v,v^{-1}]$.
\end{prop}

\begin{proof}
This follows by applying \eqref{commutative}, Proposition \ref{phi-2p} and Theorem ~\ref{j-RLconj}.
The detail is completely analogous to the proof of Proposition~\ref{prop:slsch} and hence skipped.
\end{proof}

\begin{rem}
Theorem ~\ref{j-RLconj} provides a strong evidence for 
a possible functor realization of the transfer map $\phi_{d+n, d}^\jmath$ (cf. \cite[Remark 7.10]{M12}).
In light of \cite{L99, SV00}, it is interesting to see if $\phi_{d+n, d}^\jmath$ (and hence $\phi_d^\jmath$) 
sends each canonical basis element to a canonical basis element or zero,
improving Theorem ~\ref{j-RLconj} and Proposition~\ref{prop:j-slsch}; compare with Remark~\ref{rem:CBtoCB}. 
\end{rem}

Recall there is a Schur-type $(\Scqj(n,d), \bold{H}_{B_d})$-duality on $\mathbb V^{\otimes d}$  \cite{G97, BW13},
where $\mathbb V$ is $n$-dimensional, and
this duality can be completely realized geometrically \cite{BKLW}.
Denote by $\bold{B}^\jmath(n^d)$ the $\jmath$-canonical basis of $\mathbb V^{\otimes d}$ constructed in \cite{BW13}.  
These canonical bases on $\mathbb V^{\otimes d}$ as well as on $\Scqj(n,d)$ are realized in \cite{BKLW} as simple perverse sheaves,
and the action of $\Scqj(n,d)$ on $\mathbb V^{\otimes d}$ is realized in terms of a convolution product. 
Hence we have the following positivity.

\begin{prop}
 \label{prop:j-sch}
The action of  $\Scqj(n,d)$ on $\mathbb V^{\otimes d}$ with respect to the corresponding $\jmath$-canonical bases is positive in the following sense:
for any canonical basis element $a$ of $\Scqj(n,d)$ and any $b \in \bold{B}^\jmath(n^d)$, we have
$$
a * b = \sum_{b' \in \bold{B}^\jmath(n^d)} D_{a,b}^{b'} \; b', \qquad \text{ where } D_{a,b}^{b'} \in \mathbb N[v,v^{-1}].
$$
\end{prop}

We obtain a natural action of $\Ujdsl$ on $\mathbb V^{\otimes d}$ by composing the action of $\Scqj(n,d)$ on $\mathbb V^{\otimes d}$
with the map $\phi_d^\jmath: \Ujdsl \rightarrow \Scqj(n,d)$. 
As a corollary of Propositions~\ref{prop:j-slsch} and \ref{prop:j-sch} we have the following positivity  
(which is a special case of a conjectural positivity property of the canonical basis for general tensor product modules
\cite{BW13}).

\begin{cor}
\label{cor:positive}
The action of $\Ujdsl$  on  $\mathbb V^{\otimes d}$ with respect to the corresponding $\jmath$-canonical bases is positive. 
\end{cor}

\subsection{Compatibility of canonical bases $\dot\B(\sll_{m})$ and $\Bjs$}
\label{sl-coideal}

Given integers $k, m$  with $0\le2 m \le  n$, we recall  $\tau^k_{m, n}$ from (\ref{tau-mn}). 
Fix an $m$-tuple of integers ${\bf k}  =(k_0, k_1,\ldots, k_{m-1})$.
We define an imbedding 
$\overline \tau^{k_d}_{m, n}: \dot{\U} (\sll_m) \langle \overline d \rangle  \to \Ujdsl \langle \overline{d + k_d(n-2m)} \rangle $, for $0\le d<m$, to be the composition
\begin{equation}
 \label{eq:tauk}
 \dot{\U} (\sll_m) \langle \overline d \rangle
\overset{\wp_{d}^{-1}}{\longrightarrow}
\dot{\U} (\gl_m)   \langle d \rangle
\overset{\tau^{k_d}_{m,n}}{\longrightarrow}
\Ujdgl \langle 
d + k_d(n-2m) \rangle 
\overset{\wp_{\jmath}}{\longrightarrow}
 \Ujdsl \langle \overline{d + k_d(n-2m)} \rangle.
\end{equation}
These $\overline \tau^{k_d}_{m, n}$ for all $d$ can be combined into a homomorphism $\overline \tau^{{\bf k}}_{m, n}: \dot{\U} (\sll_m)  \to \Ujdsl$.
We recall $\overline{\Theta}^m$ from \eqref{barTh}, which is understood in this subsection to consist of $m\times m$ matrices.

\begin{prop}
\label{prop:sameCB}
Retaining the notations above, we have 
$\overline \tau^{{\bf k}}_{m, n} \big(\dot\B(\sll_{m}) \big) \subseteq \Bjs$. More precisely,
if $b_{\overline A} \in\dot \B(\sll_m)$ for $\overline{A} \in \overline{\Theta}^m$, then 
$\overline \tau^{{\bf k}}_{m, n} (b_{\overline A}) = b_{\widehat{A'}}$, where $A' =\tau^{k_d}_{m,n} (A)$ if $|A| =d$.
\end{prop}

\begin{proof}
We have the following commutative diagram:
\[
\begin{CD}
\dot{\U} (\gl_m)   \langle d \rangle
@> \tau^{k_d}_{m,n} >>
\Ujdgl \langle d + k_d(n-2m) \rangle 
 \\
 @V\xi_{2l} VV @V\xi^{\jmath}_{2l}VV\\
\dot{\U} (\gl_m)   \langle d + 2lm \rangle
@> \tau^{k_d + l }_{m,n}>>
\Ujdgl \langle
d + k_d(n-2m)+ ln 
 \rangle 
\end{CD}
\]
Let $\overline{A} \in \overline{\Theta}^m$.
Pick the preimage (an $m\times m$ matrix) $A$ of $\overline{A}$ with  $0\le |A| <m$, and set $d=|A|$. 
Recall from \eqref{eq:wpxi} and \eqref{eq:wpj} that $ \wp \circ \xi_{2l} = \wp$
and $ \wp_\jmath \circ \xi_{2l}^\jmath = \wp_\jmath,$ for $l \in \Z.$
It follows from these identities, \eqref{eq:tauk}, and the above commutative diagram that 
$\overline \tau^{k_d}_{m, n} 
=\wp_\jmath \circ \tau^{k_d + l }_{m,n} \circ \wp_{d+2lm}^{-1}$. 
Hence 
applying \cite[Proposition 7.8]{M12}, Lemma ~\ref{lr},  and Proposition ~\ref{prop:xiCB} in a row give us (for $l \gg 0$)
$$
\overline \tau^{k_d}_{m, n}  (b_{\overline A}) 
=\wp_\jmath \circ \tau^{k_d + l }_{m,n} \circ \wp_{d+2lm}^{-1} (b_{\overline A}) 
=\wp_\jmath \circ \tau^{k_d + l }_{m,n} (\ \! {}^{\bold a}\{\overline {}_{2l}A\}) 
=\wp_\jmath  (\{\tau^{k_d + l }_{m,n} (\overline {}_{2l}A) \}) 
=b_{\widehat{A'}},
$$
where the last identity uses the fact that $A' =\tau^{k_d}_{m,n} (A)$ and $\tau^{k_d + l }_{m,n} (\overline {}_{2l}A)$ have the same image in $\widehat{\Xi}$.
The proposition is proved. 
\end{proof}

\subsection{A positive basis for $\Ujdgl$}
\label{j-gl-cb}

Recall that the stably canonical basis of $\Ujdgl$ (and hence of $\Ujdgl \langle d\rangle$ for $d\in \Z$) does not have
positive structure constants in general by Proposition~\ref{prop:negCBj}. 
However,  one can transport the canonical basis on $\Ujdsl \langle \ov{d}\rangle$
 to $\Ujdgl \langle d\rangle$ via the isomorphism $\wp_{d,\jmath}$ in \eqref{eq:wpjd}, which has positive structure constants by Theorem~\ref{th:posCBj}. Let us denote the resulting
 {\em positive basis} (or {\em can$\oplus$nical basis}) on $\Ujdgl =\oplus_{d\in \Z} \Ujdgl \langle d\rangle$ by $\B^\jmath_{\text{pos}}(\gl_n)$. 
 By definition, the basis  $\B^\jmath_{\text{pos}}(\gl_n)$ is invariant under the shift maps $\xi_p^\jmath$ for $p\in 2\Z$.
 Summarizing we have the following.
 
 \begin{prop}
  \label{prop:gljCB+}
 There exists a positive basis $\B^\jmath_{\text{pos}}(\gl_n)$ for $\Ujdgla$ (and also for $\Ujdgl$), which is induced from the canonical basis
 for $\Ujdsla$. 
 \end{prop}

It is clear that the transition matrix between the positive basis and the stably canonical basis of $\Ujdgla$ is unitriangular.

\section{Canonical basis for $\Uidsl$ for $\nn$ even}
\label{n-even}

In this section, we shall construct  the canonical basis of $\Uidsl$ for $\nn$ even with positivity properties.
This is achieved by relating to the case of $\dot{\mbf{U}}^\jmath (\sll_n)$ for $n$ odd studied in the previous two sections with 
\[
\nn=n-1  \geq 2\quad  (\nn \mbox{ even}).
\]

\subsection{$\imath$Schur algebra $\Si$ and the transfer map $\phi^{\imath}_{d+\l, d}$}

Recall $\Scj(n, d)$ from Section ~\ref{j-Schur}.
We define $\Sci(\l, d)$ to be the $\A$-submodule of $\Scj(n, d)$ spanned by the standard basis element $[A]_d$, where
$A$ runs over the following subset of $\Xi_d$ in (\ref{Xi-d}).
\begin{align}
 \label{Xidi}
\Xi_d^{\imath} =\{ A\in \Xi_d \vert a_{\frac{\l}2 +1, j} = \delta_{\frac{\l}2 +1, j}, a_{i, \frac{\l}2 +1} = \delta_{i, \frac{\l}2 +1} \}.
\end{align}
Clearly, this is a subalgebra of $\Scj(n, d)$ over $\A$.
Note that when the parameter $v$ is  specialized at $\sqrt q$, the algebra  $\Sci(\l, d)$ 
coincides with the convolution algebra of pairs of $\l$-step partial isotropic flag in $\mbb F^{2d+1}_q$ equipped with a fixed non-degenerate symmetric bilinear form. 
Moreover, the subset $\{ \{A\}_d | A\in \Xi^{\imath}_d\}$ of the canonical basis of $\Scj(n, d)$ is an $\A$-basis of $\Sci (\l, d)$.
Let
$$
\Si   = \Q(v) \otimes_{\A} \Sci(\l,d).
$$
Recall from (\ref{jphi}), we have an algebra homomorphism
$\Scqj(n+d, d) \to \Scqj(n, d) \otimes \Scq (n, n)$. By restricting to $\Si$, we obtain an algebra homomorphism
\[
\Delta^{\imath}: \mbf S^{\imath}(\nn, \nn+d) \to \Si \otimes \Scq(\l, \l),
\]
where we identify $\Scq( \l, \l)$ with the subalgebra in $\Scq(n,n)$ spanned by the elements $[A]$ whose entries in the $(\frac{\l}2+1)$st rows and columns are zero. 
We refer to ~\cite[Lemma 5.1.1]{FL15} for a more explicit construction of $\Delta^{\imath}$, which is denoted $\widetilde \Delta^{\imath}$ therein.
Recall the sign homomorphism $\chi_n$ from  \eqref{sign}, and we define the transfer map $\phi^{\imath}_{d + \l, d}: \Scq^{\imath}(\l, d+\l) \to \Si$ to be the composition
\[
\begin{CD}
\phi^{\imath}_{d + \l, d}:  \Scq^{\imath}(\l, d + \l) @>\Delta^{\imath} >> \Si \otimes \Scq(\l, \l) @> 1 \otimes \chi_{\l} >> \Si.
\end{CD}
\]
We set 
$$
\mbb I = I - E_{n+1, n+1}.
$$
By ~\cite[Corollary 5.1.4]{FL15}, we have
\begin{align}
\label{i-phi}
\phi^{\imath}_{d+\l, d}  (\{X\}_{d+ \l}) =
\begin{cases}
\{X - 2 \I\}_d, & \mbox{if} \ X - 2 \I \in \Xi_d^{\imath}, \\
0, & \mbox{otherwise}.
\end{cases}
\end{align}
for all matrices $X \in \Xi^{\imath}_d$ such that either one of the following matrices is diagonal: 
$X$, $X - E^{\theta}_{\frac{\l}2, \frac{\l}2+2}$, $X - a E^{\theta}_{i+1, i}$ or 
$X - a E^{\theta}_{i, i+1}$ where  $ a\in \mathbb N$, $1\leq i \leq \frac{\l}2 -1$.

\begin{rem}
As we will show that if $X$ is chosen such that  $X - aE^{\theta}_{\frac{\l}2, \frac{\l}{2}+2}$ is diagonal for $a \geq 2$, the formula (\ref{i-phi}) fails to be true.
This makes the construction of canonical basis for $\Uidsl$ more subtle than that of $\Ujdsl$.
This subtlety boils down to the detailed analysis of the rank-one transfer map, which is the main topic of the following subsection.
\end{rem}

\subsection{The  transfer map on $\mathbf{S}(2,d)$}
\label{rank-one-iphi}

In this subsection, we set $\l =2$ (hence $r=1$) and consider the rank-one transfer map $\phi^{\imath}_{d, d-2}: \Scq^{\imath} (2, d) \longrightarrow \Scq^{\imath}(2, d-2)$.
For convenience, we set
\begin{align}
\label{Aab}
A_{a,b} = 
\begin{pmatrix}
a & 0 & b \\
0 & 1 & 0 \\
b & 0 & a
\end{pmatrix}.
\end{align}
Thus if $A_{a, b} \in \Xi_d$, we have $a+b =d$. 
In this subsection we drop the index $d$ to write $[A_{a, b}]$ and $\{A_{a, b}\}$ for $[A_{a, b}]_d$ and $\{A_{a, b}\}_d$, respectively. 
We set
$[A_{a,b}]=0, \ \mbox{if}\ a<0 \ \mbox{or}\ b <0.$

\begin{lem}
\label{phiM}
For all $a, b\in \mathbb N$ such that $a+b =d$,  we have
$$
\phi^{\imath}_{d, d-2} ([A_{a,b}]) = [A_{a-2,b}] + (v^{-a+1} - v^{-a-1}) [A_{a-1,b-1}] - v^{-2a-1} [A_{a, b-2}]. 
$$
\end{lem}

\begin{proof}
We shall prove the lemma by induction on $b$. 
When $b=0$, the statement follows from the definition of $\phi^{\imath}_{d, d-2}$.  

Let $b \in \mathbb N$, and we assume the formula in the lemma is proved for $\phi^{\imath}_{d, d-2} ([A_{a,b'}])$,
for all $b' \le b$ and all $a$. 
We set
$\mbf t_d = \{ A_{d-1, 1}\}.$
Recall from ~\cite[Lemma~ A.13]{BKLW} that we have
\begin{align}
\label{t*M}
\mbf t_d *  [A_{a, b}] =
v^{-a+b} [A_{a,b}] + v^b \overline{[b+1]} [A_{a-1, b+1}] + v^{b-1} \overline{[a+1]} [A_{a+1,b-1}].
\end{align}
By induction and using (\ref{t*M}), we have
\begin{align}
\label{phi-t-m}
\begin{split}
\phi^{\imath}_{d, d-2}  & ( \mbf t_d * [A_{a, b}] )  
%
 = \phi^{\imath}_{d, d-2}  ( 
v^{-a+b} [A_{a,b}] + v^b \overline{[b+1]} [A_{a-1, b+1}] + v^{b-1} \overline{[a+1]} [A_{a+1,b-1}]
) \\
& = ( v^{-a + b +2} + v^{b-1} \overline{[b]} (v^{-a+1} - v^{-a-1})) [A_{a-2,b}] + v^b \overline{[b+1]} [A_{a-3, b+1}] \\
&  + \left( v^{b-1}\overline{[a-1]} + (v^{-a+1} - v^{-a-1}) v^{-a+b} - v^{-2a+b-3}\overline{[b-1]} \right) [A_{a-1,b-1}] \\
& + \left(v^{b-2}\overline{[a]} ( v^{-a+1} - v^{-a-1}) - v^{-3a+b-3} \right ) [A_{a, b-2}] - v^{-2a+b -4} \overline{[a+1]} [A_{a+1,b-3}].
\end{split}
\end{align}
By combining (\ref{t*M}) and (\ref{phi-t-m}), we have
\begin{align*}
 & v^b  \overline{[b+1]}   \phi^{\imath}_{d, d-2} ([ A_{a-1,b+1}])  \\
& = \phi^{\imath}_{d, d-2}  ( \mbf t_d * [A_{a, b}] )  - v^{-a+b} \phi^{\imath}_{d, d-2} (  [A_{a, b}] ) 
-  v^{b-1} \overline{[a+1]} \phi^{\imath}_{d, d-2} (  [A_{a+1, b-1}] ) \\ 
& = 
 v^b \overline{[b+1]}  [A_{a-3, b+1}] 
 + \left ( v^{-a+b+2} - v^{-a+b} + v^{b-1} \overline{[b]} ( v^{-a+1} - v^{-a-1}) \right) [A_{a-2,b}]\\
& + \left( v^{b-1} \overline{[a-1]} - v^{-2a_b-3} \overline {[b-1]} - v^{b-1} \overline{[a+1]} \right) [A_{a-1,b-1}]\\
& + \left( v^{b-2} \overline{[a]} (v^{-a+1} - v^{-a-1}) - v^{-3a+b-3} + v^{-3a+b-1} - v^{b-1} \overline{[a+1]} ( v^{-a} - v^{-a-2}) \right) [A_{a, b-2}]\\
&
= 
v^b \overline{[b+1]} 
\left( [A_{a-3, b+1}]  + (v^{-a+2} - v^{-a} ) [A_{a-2, b}] - v^{-2a+1} [A_{a-1, b-1}] \right) .
\end{align*}
Thus we have
\[
 \phi^{\imath}_{d, d-2} ( [A_{a-1,b+1}]) =
 [A_{a-3, b+1}]  + (v^{-(a-1)+1} - v^{-(a-1)-1} ) [A_{a-2, b}] - v^{-2(a-1)-1} [A_{a-1, b-1}].
\]
The lemma is proved.
\end{proof}

\begin{prop}
\label{i-transfer-cb}
We have 
\begin{align}
\phi^{\imath}_{d, d-2} ( \{ A_{a, b}\} ) =
\begin{cases}
 \{A_{a-2, b}\}, &  \mbox{if} \  a\geq 2,\\
 \{A_{0, b-1}\}, & \mbox{if} \ a=1,\\
 0, & \mbox{if} \ a=0.
\end{cases}
\end{align} 
\end{prop}

\begin{proof}
The coefficients of $[A_{a-1, b-1}]$ and $[A_{a, b-2}]$ in the expansion of $\phi^{\imath}_{d, d-2} ([A_{a,b}])$ 
are in $v^{-1} \mbb Z[v^{-1}]$ for $a \geq 2$, by Lemma ~\ref{phiM}. 
Meanwhile, $A_{a', b'} \preceq A_{a, b}$ if and only if $a' \geq a$. So  we have
\begin{align}
\label{pab}
\phi^{\imath}_{d, d-2} (\{ A_{a,b} \})
\in [A_{a-2, b}] + \sum_{i=1}^b v^{-1}\mbb Z[v^{-1}]  [A_{a-2+i, b-i}].
\end{align}
Since $\phi^{\imath}_{d, d-2} (\{ A_{a,b} \})$ is bar invariant, we conclude that $\phi^{\imath}_{d, d-2} (\{ A_{a,b} \}) =\{A_{a-2, b}\}$ if $a\geq 2$.

For $a=1$, we write
\[
\{A_{1,b}\} = [A_{1,b}] + \sum_{i=1}^b Q_{i} [A_{1+i, b-i}], \quad \mbox{for some} \ Q_i \in v^{-1} \mbb Z[v^{-1}].
\]
Thus
\begin{equation} 
 \label{A1}
\phi^{\imath}_{d, d-2} (\{ A_{1,b} \}) = 
(1 - v^{-2}) [A_{0, b-1}] - v^{-3} [A_{1, b-2}] + \sum_{i=1}^b Q_i   \phi^{\imath}_{d, d-2} ([A_{1+i, b-i}]).
\end{equation}
By Lemma ~\ref{phiM}, 
the coefficient of $[A_{0, b-1}]$ on the RHS of \eqref{A1} is in $1 + v^{-1} \mbb Z[v^{-1}]$ and the coefficients of $[A_{1+i, b-i}]$
 on the RHS of \eqref{A1} for $i\geq 0$ are in $v^{-1}\mbb Z[v^{-1}]$. 
Now since $\phi^{\imath}_{d, d-2} (\{ A_{a,b} \})$ is bar invariant,   the coefficient of $[A_{0, b-1}] $ must be $1$, and
we have $\phi^{\imath}_{d, d-2} (\{ A_{1 , b} \}) =\{ A_{0, b-1}\}$.

Now Lemma~\ref{phiM} for $a=0$ gives us  $\phi^{\imath}_{d, d-2} ([A_{0,b}]) = - v^{-1} [A_{0, b-2}]$. 
A similar analysis as for $a=1$ shows that the expansion 
of $\phi^{\imath}_{d, d-2} (\{ A_{0 , b} \})$ with respect to the standard  basis $[A_{a,b}]$ 
have all coefficients in $v^{-1}\mbb Z[v^{-1}]$. This  yields $\phi^{\imath}_{d, d-2} (\{ A_{0 , b} \})=0$ due to its bar-invariance property.

The proposition is  proved.
\end{proof}

In Section ~\ref{rank-one-iCB}, we will give an explicit formula of the canonical basis in $\Scq^{\imath}(2, d)$ in terms of standard basis.

\subsection{Hybrid monomial basis for $\Si$}
\label{HMB}

Now we consider $\mbf S^{\imath}(\l, d)$ for a general even integer $\l$. 
Recall the monomial basis $\{ {}_d\texttt{M}_A | A \in \Xi^{\imath}_d\} $ of $\mbf S^{\imath} (\l, d)$ from ~\cite[Proposition~ 5.6]{BKLW}; 
for notation $\Xi^{\imath}_d$ see \eqref{Xidi}. 
This is a subset of the monomial basis $\{ {}_d \texttt{M}_A\}$ in $\Sj (n, d)$ \cite[(3.25)]{BKLW} (denoted by $m_A$ therein) used in Section~\ref{sec:Sjtransfer}, 
and $_d\texttt{M}_A$ is a  monomial  in $[X]_d$ where either $X -a E^{\theta}_{i, i+1}$, or $X - a E^{\theta}_{i+1, i}$, for all $1\leq i \leq \frac{\l}2 -1$, is diagonal or
a twin product  $[X_1]_d * [Y_1]_d$, 
where the matrices  $X_1 - a E^{\theta}_{\frac{\l}2, \frac{\l}2 +1}$ and $Y_1 - aE^{\theta}_{\frac{\l}2 +1, \frac{\l}2}$ for some $a\in \mathbb N$  are diagonal 
and $\co (X_1) = \ro (Y_1)$.
Recall that  the subset $\{\{A\}_d | A \in \Xi^{\imath}_d\}$ of the canonical basis of $\Sj(n,d)$ forms a basis for $\Si$ and 
for the  twin pair $[X_1]_d *[Y_1]_d$, we have  
\begin{align}
\label{feE}
[X_1]_d * [Y_1]_d = \{E_{\frac{\l}2, \frac{\l}2+2}^{\theta}(a)\}_d + \mbox{lower terms } \in \Si.
\end{align}
Here $E_{\frac{\l}2, \frac{\l}2+2}^{\theta}(a)$ is the unique matrix defined by the conditions:
$\co \big(E_{\frac{\l}2, \frac{\l}2+2}^{\theta}(a) \big) = \co(Y_1)$ and $E_{\frac{\l}2, \frac{\l}2+2}^{\theta}(a) -  a E^{\theta}_{\frac{\l}2, \frac{\l}2+2}$ is diagonal. 

\begin{Def}
The hybrid  monomial $_d\texttt{M}^{\imath}_A$ is obtained from $_d \texttt{M}_A$ by replacing  the twin product $[X_1]_d * [Y_1]_d$
in $_d\texttt{M}_A$  by the leading term  $\{E_{\frac{\l}2, \frac{\l}2+2}^{\theta}(a)\}_d $ in (\ref{feE}).
\end{Def}

The following properties of the hybrid monomials $_d \texttt{M}^{\imath}_A$ are the main reasons  to introduce them.

\begin{prop}
\label{Mi}
The following properties hold for a hybrid monomial ${}_d\texttt{M}_A^{\, \imath}$ (where $A \in \Xi^{\imath}_d$):
\begin{enumerate}
\item $\overline{_d\texttt{M}^{\, \imath}_A} = \ _d\texttt{M}^{\, \imath}_A$,

\item $_d\texttt{M}^{\, \imath}_A = \{A\}_d +$ lower term,  

\item the set $\{_d\texttt{M}^{\, \imath}_A \big \vert A \in \Xi^{\imath}_d\}$ forms a basis of $\Sci(\l, d)$,

\item 
$\phi^{\imath}_{d, d-{\nn}} ( _d \texttt{M}^{\, \imath}_A) = \ _{d-{\nn}} \texttt{M}^{\, \imath}_A$, whenever $a_{ii} \gg 0$ for all $i\in [1, \frac{\l}2]$.
\end{enumerate}
\end{prop}

\begin{proof}
Items (1)-(3) follow readily by construction. 
Since ${}_d\texttt{M}_A^{\imath}$ is obtained by modifying the factors in $_d\texttt{M}_A$  at finitely many places, 
it is clear that we can add $p \mbb I$ for $p$ large enough to $A$ such that all twin product $[X_1]_d * [Y_1]_d$
appearing in $_d \texttt{M}_{A+p\mbb I}$ have their $(\frac{\nn}2, \frac{\nn}2)$th entries $\geq 2$.  
Item (4) now follows from the analysis of the rank one transfer map in Section ~\ref{rank-one-iphi}. 
\end{proof}

\subsection{The modified quantum coideal subalgebras $\Uidgl$ and $\Uidsl$}

Recall the algebra $\Kj$ from Section ~\ref{j-Schur}.
This algebra has a standard basis $[A]$ parameterized by the set $\tilde \Xi$ in (\ref{txi}). 
Let $\Kj_1$ be the subalgebra of $\Kj$ spanned by the standard basis $[A]$ in $\tilde \Xi$ such that
$\ro (A)_{\frac{\l}2 +1} = \co (A)_{\frac{\l}2 +1} =1$.
Let $\mathcal J_1$ be the ideal of $\Kj_1$ spanned by $[A]$ for all $A\in \tilde \Xi^{\imath}$ such that $a_{\frac{\l}2 + 1, \frac{\l}2 +1} <0$.
Let $\tilde \Xi^{\imath}$ be the subset of $\tilde \Xi$ consisting of matrices $A$ defined by $a_{\frac{\l}2 +1, j} = \delta_{\frac{\l}2 +1, j}$
and $a_{i, \frac{\l}2+1} = \delta_{i, \frac{\l}2+1}$ for all $i, j$.

We set $\Ki$ be the quotient of $\Kj_1$ by $\mathcal J_1$.
It is shown in ~\cite[Appendix~A.3]{BKLW} that $\Ki$ admits a monomial basis $\texttt{M}_A + \mathcal J_1$, 
a standard basis $[A] + \mathcal J_1$, and a canonical basis $\{A\} + \mathcal J_1$, for all $A \in \tilde \Xi^{\imath}$.
Furthermore, it is shown in ~\cite[Proposition A.11]{BKLW} that $\Ki$
is isomorphic to the modified quantum coideal subalgebra $\Uidgl$  of the quantum algebra $\mathbf U (\mathfrak{gl}_n)$.
We shall identify $\Ki$ with $\Uidgl$.
Recall that the algebra $\Uidgl$ is an associative  $\Qq$-algebra generated by the symbols $1_{\lambda}$, $\ibe{i}1_{\lambda}$,
$1_{\lambda}\ibe{i}$, $\ibff{i}1_{\lambda}$,  $1_{\lambda}\ibff{i}$, $ t 1_{\lambda}$, and $1_{\lambda} t$, for $i = 1, \dots, \frac{\l}2 -1 $ and ${\lambda} \in \mbb Z^{\imath}_{\l} : =\{ \lambda \in \mbb Z^{\jmath}_n | \lambda_{\frac{\l}2+1}=1\}$,
subject to the following relations \eqref{i:align:doublestar}:
for $i,j = 1, \dots, \frac{\l}2 -1$,
$\la, \la' \in \mbb Z^{\imath}_{\l}$,  and  for  $x, x'\in \{1, \ibe{i}, \ibe{j}, \ibff{i}, \ibff{j}, t\},$
\begin{eqnarray}
\label{i:align:doublestar}
\left\{
 \begin{array}{rll}
x 1_{\lambda}  1_{\lambda'} x' &= \delta_{\la, \la'} x 1_{\lambda}  x', &  \\
\ibe{i} 1_{\lambda} &= 1_{\lambda -\alpha_i}   \ibe{i}, & \\
\ibff{i} 1_{\lambda} &=  1_{\lambda +\alpha_i}  \ibff{i},  & \\
t 1_{\lambda} &= 1_{\lambda} t, \\
 \ibe{i} 1_{\lambda}\ibff{j} &=  \ibff{j}  1_{\lambda-\alpha_i -\alpha_j}   \ibe{i}, &\text{if } i \neq j,  \\
 \ibe{i} 1_{\lambda} \ibff{i} &=\ibff{i} 1_{\lambda -2\alpha_i}   \ibe{i}  + [ \lambda_{i+1} - \lambda_{i} ]  1_{\lambda-\alpha_i},
 \\
(\ibe{i}^2 \ibe{j} +\ibe{j} \ibe{i}^2) 1_{\lambda} &= [ 2 ]   \ibe{i} \ibe{j} \ibe{i} 1_{\lambda}, & \text{if }  |i-j|=1, \\
( \ibff{i}^2 \ibff{j} +\ibff{j}\ibff{i}^2) 1_{\lambda} &= [ 2 ]  \ibff{i} \ibff{j} \ibff{i} 1_{\lambda} , &\text{if } |i-j|=1,\\
 \ibe{i} \ibe{j} 1_{\lambda} &= \ibe{j} \ibe{i} 1_{\lambda} ,   & \text{if } |i-j|>1, \\
\ibff{i} \ibff{j} 1_{\lambda}  &=\ibff{j} \ibff{i} 1_{\lambda} ,   & \text{if } |i-j|>1, \\
 t \ibff{i} 1_{\lambda} &= \ibff{i} t 1_{\lambda}, &\text{if } i \neq \frac{\l}2 -1,\\
   (t^2 \ibff{\frac{\l}2 -1} + \ibff{\frac{\l}2 -1} t^2) 1_{\lambda} &= \big([ 2 ]   t \ibff{\frac{\l}2 -1} t + \ibff{\frac{\l}2 -1}\big) 1_{\lambda} ,\\
  (\ibff{\frac{\l}2 -1}^2 t + t \ibff{\frac{\l}2 -1}^2) 1_{\lambda}
    &= [ 2 ]   \ibff{\frac{\l}2 -1} t \ibff{\frac{\l}2-1} 1_{\lambda},   \\
    t \ibe{i} 1_{\lambda} &= \ibe{i} t 1_{\lambda}, &\text{ if } i \neq \frac{\l}2-1,\\
    (t^2 \ibe{\frac{\l}2 -1} + \ibe{\frac{\l}2 -1} t^2) 1_{\lambda} &= \big([2]  t \ibe{\frac{\l}2 -1} t + \ibe{\frac{\l}2 -1} \big) 1_{\lambda},\\
 (\ibe{\frac{\l}2 -1}^2 t + t \ibe{\frac{\l}2 -1}^2) 1_{\lambda}
   &= [2]  \ibe{\frac{\l}2-1} t \ibe{\frac{\l}2 -1} 1_{\lambda}.
   \end{array}
    \right.
\end{eqnarray}
Here $\lambda \pm \alpha_i$ are the short hand notations introduced in Section ~\ref{j-Schur}.
To simplify the notation, we shall write $x_1 1_{\lambda^1} \cdot x_2 1_{\lambda^2} \cdots x_l 1_{\lambda^l} = x_1 x_2 \cdots x_l 1_{\lambda^l}$,
if the product is not zero.

We define an equivalence relation $\approx$ on $\mbb Z^{\imath}_{\l}$  by setting $\lambda \approx \lambda'$ if and only if 
$\lambda - \lambda' = a \I$ for some $a\in 2\mbb Z$.
Let $\hat {\mbb Z}^{\imath}_{\l}$ be the set $\mbb Z^{\imath}_{\l} / \approx$ of equivalence classes.
Let $\Uidsl$ be the algebra defined in the same fashion as $\Uidgl$ with the parameter set $\mbb Z^{\imath}_{\l}$ replaced by 
$\hat{\mbb Z}^{\imath}_{\l}$.
Similar to $\Ujdgl$, the algebras $\Uidgl$ and $\Uidsl$ admit the following decompositions.
\begin{equation*}
\begin{split}
\Uidgl & =\bigoplus_{d \in \Z}  \Uidgl \langle d \rangle, \\
\Uidsl & =\bigoplus_{\bar{d} \in \Z/ \l \Z} \Uidsl \langle \bar{d}\rangle,
\end{split}
\end{equation*}
where $\Uidgl \langle d \rangle$ is spanned by elements of the form $1_\la u 1_\mu$ with $|\mu| =|\la| =2d+1$ and $u\in \Uidgl$, 
and 
$\Uidsl \langle \bar{d} \rangle$ is spanned by $1_{\ov{\mu}}  \Uidsl 1_{\ov{\la}}$, where $\ov{\mu}, \ov{\la} \in \hat{\mbb Z}^{\imath}_{\l}$,
$|\ov{\mu}| \equiv |\ov{\la}| \equiv 2d+1 \mod 2 \l$.

We have the following commutative diagram similar to (\ref{commutative}):
\begin{align} \label{i-commutative}
\begin{split}
\xymatrix{ 
&&
\Uidgl  \ar@<0ex>[d]^{\wp_\imath} \ar@<0ex>[lldd]_{\Psi_{d+ \l}^\imath}  \ar@<0ex>[rrdd]^{\Psi_d^\imath}
&&\\
&& \Uidsl \ar@<0ex>[lld]^{\phi_{d +\l}^\imath}  \ar@<0ex>[rrd]_{\phi_d^\imath} &&\\
\mbf S^{\imath} (\l, d+ \l)  \ar@<0ex>[rrrr]^{\phi_{d+ \l, d}^\imath  }
&&&&  \Si }
\end{split}
\end{align}
Here the homomorphisms $\phi^{\imath}_d$ and $\wp_{\imath}$ are defined in a similar way as 
$\phi^{\jmath}_d$ and $\wp_{\jmath}$ in (\ref{commutative}) respectively,   but with $I$ replaced by $\I$.

\subsection{Inner product on $\Uidsl$}

Let $\langle -,- \rangle_{\imath, d}$ be the bilinear form on the $\imath$Schur algebra $\Si$  
obtained from the bilinear form $\langle -, - \rangle_d$ on $\Sj(n, d)$ by restriction, thanks to $\Si \subset \Sj(n, d)$. 
We define a family of bilinear forms $\langle -,- \rangle_{\imath, d}$ on $\Uidsl$ by pulling back the one on the Schur algebra level via 
$\phi_d^{\imath}$ in (\ref{i-commutative}), i.e., 
$\langle x, x' \rangle_{\imath, d} = \langle \phi^{\imath}_d (x), \phi^{\imath}_d (x') \rangle_{\imath, d}$ for $x, x' \in \Uidsl$.
We shall study the behavior of these bilinear forms  as $d$ tends to infinity.
We need the following analogue of  ~\cite[Lemma 4.2]{M12}.

\begin{lem}
 \label{lem:multstable}
Let $A_i$ ($1\leq i \leq k$)   be matrices such that either $A_i - E_{h+1,h}^{\theta}$, $A_i - E_{h, h+1}^{\theta}$,  ($h\in [1,\frac{\l}2 -1]$),  or $A_i - E_{\frac{\l}2, \frac{\l}2+2}^{\theta}$  is diagonal.  Let $A \in \tilde \Xi^{\imath}$ with $|A|=d$. Then there exists matrices $Z_1, Z_2, \cdots, Z_m \in \tilde \Xi^{\imath}$, and $G_1(v, u) , \cdots, G_m (v, u) \in \mbb Q(v)[u]$ and an integer $p_0 \in \mbb Z$ such that
$$
\{ A_1 + p\mbb I\}_{d + p\nn}  * \{A_2 + p \mbb I \}_{d + p\nn}  * \cdots * \{ A_k + p\mbb I \}_{d + p\nn}  * [A + p \mbb I]_{d + p\nn} 
= \sum_{i=1}^m G_i (v, v^{-p}) [Z_i + p\mbb I]_{d + p \nn},
$$ 
for all even integer  $p \geq p_0$.
\end{lem}

\begin{proof}
The proof follows  the arguments of ~\cite[Lemma 4.2]{M12} and ~\cite[Lemma~ A.1]{BKLW}, 
except that we need to take care of the new case when $k=1$ and $A_1 - E_{\frac{\l}2, \frac{\l}2+2}^{\theta} $ is diagonal.  
In this case, we need the following multiplication formula in $\Si$ from ~\cite[Lemma A.13]{BKLW} for the generator $\bt_d = \sum \{X\}_d$ where $X$ runs over all matrices in $\Xi^{\imath}_d$ such that $X - E^{\theta}_{\frac{\l}2, \frac{\l}2 +2}$ is diagonal.  
(That $\bt_d$ can be written in such a form is due to ~\cite[Lemma 5.5]{BKLW}.)
For any $A \in \Xi_d^{\imath}$,  we have
\begin{align}
\label{A.13}
\bt_d * [A]_d = \sum_{1 \leq j \leq n} v^{\sum_{j \geq p} a_{\frac{\l}2 + 2, p} - \sum_{j > p} a_{\frac{\l}2, p} - \sum_{p > \frac{\l}2+1}  \delta_{j, p} } \overline{[a_{\frac{\l}2+2, j} +1]}  [A - E_{\frac{\l}2, j}^{\theta} + E^{\theta}_{\frac{\l}2 + 2, j} ]_d.
\end{align}
Then we set $Z_j = A - E_{\frac{\l}2, j}^{\theta} + E^{\theta}_{\frac{\l}2+2, j} $ and
\[
G_j(v,u) = v^{\sum_{j \geq p} a_{\frac{\l}2 + 2, p} - \sum_{j > p} a_{\frac{\l}2, p} - \sum_{p > \frac{\l}2+1} \delta_{j, p} } 
u^{\delta_{j, \frac{\l}2 +1}} \frac{ v^{-2(a_{\frac{\l}2+2, j} +1)} u^{2 \delta_{j, \frac{\l}2+2}} -1}{v^{-2} -1}.
\]
The lemma now follows by induction. 
\end{proof}

\begin{rem}
Note that in the multiplication formula in Lemma~\ref{lem:multstable}, 
the canonical basis elements are used instead of the standard basis elements, which are the same for all generators except $\bt_d$.
\end{rem}

We are ready to state the asymptotic behavior of the form $\langle \cdot, \cdot \rangle_{\imath, d}$.

\begin{prop}
\label{asymptotic-form}
As $p$ goes to infinity, the limit $\lim_{p\to \infty} \langle x , x' \rangle_{\imath, d + p \l}$, for all $x, x' \in \Uidsl$, converges in $\mbb Q((v^{-1}))$ to an element in $\mbb Q(v)$.
\end{prop}

\begin{proof}
The proof is similar to ~\cite{M12}. We need the adjointness of the bilinear form $\langle \cdot, \cdot \rangle_{\imath, d}$, 
which is inherited from that of $\langle \cdot, \cdot \rangle_{d}$, and 
in particular we have
\[
\langle \bt_d  * \{A\}_d, \{ B\}_d \rangle_{\imath, d} 
= \langle \{A\}_d, \bt_d * \{B\}_d \rangle_{\imath, d}
\]
from ~\cite[Corollary 3.15]{BKLW}.
The only difference from ~\cite{M12} is that we work in a larger ring $\mbb Q(v) [u]$. 
Now suppose that $G(v, u) = \sum_{i=0}^n a_i u^i$ where $a_i \in \mbb Q(v)$. Then we have
\[
G(v, v^{-p}) = \sum_{i=0}^m a_i v^{-pi},
\]
which implies that $\lim_{p\to \infty} G(v, v^{-p}) = a_0$ in $\mbb Q((v^{-1}))$.
\end{proof}


Similar to the form $\langle \cdot, \cdot \rangle_{\jmath}$, we define a bilinear form $\langle -,-\rangle_{\imath}$ on $\Uidsl$ (independent of $d$) by letting
\begin{align}
\langle x ,  x' \rangle_{\imath} = \sum_{d=0}^{{\nn}-1} \lim_{p\to \infty} \langle x, x' \rangle_{\imath, d + p \l}, 
\quad \forall x, x' \in \Uidsl.
\end{align}

\subsection{Hybrid monomial basis in $\Uidgl$}

As for the construction of the canonical basis for  $\Ujdsl$, we need a version of monomial basis on $\Uidgl$ which enjoys similar properties in 
Proposition ~\ref{Mi} in order to construct the canonical basis of $\Uidsl$. 
%
We lift the basis $\{{}_d \texttt{M}_A^{\imath} \}$ of $\Si$ to a basis of  $\Uidgl$ with the desired properties.
The procedure is exactly the same as used in the construction of the basis $\{ {}_d \texttt{M}^{\imath}_A\}$ for the $\imath$Schur algebras in Section ~\ref{HMB}. 
More precisely, recall a monomial basis $\{\texttt{M}_A \vert A \in \tilde \Xi^{\imath} \}$ for $\Ki \equiv \Uidgl$ was constructed in \cite[Appendix~A]{BKLW} 
by lifting the (usual) monomial basis $\{{}_d \texttt{M}_A \}$ for $\imath$Schur algebras.
we form the hybrid  monomial $\texttt{M}^{\imath}_A$ from $ \texttt{M}_A$  by substituting  any twin product $[X_1] * [Y_1]$
in $\texttt{M}_A$ as in (\ref{feE}) with its leading term  $\{E_{\frac{\nn}2, \frac{\nn}2+2}^{\theta}(a)\}$  with indices $d$ dropped.

\begin{prop}
\label{Ki-Mi}
The following properties hold for a hybrid monomial $\texttt{M}_A^{\, \imath}$ with $A\in \tilde \Xi^{\imath}$:
\begin{enumerate}
\item $\overline{\texttt{M}^{\, \imath}_A} = \texttt{M}^{\, \imath}_A$,

\item $\texttt{M}^{\, \imath}_A = \{A\} +$ lower term,

\item the set $\{\texttt{M}^{\, \imath}_A | A \in \tilde \Xi^{\imath}\}$ forms a basis of $_\A\Uidgl$,

\item 
$\phi^{\imath}_d (\texttt{M}^{\, \imath}_A) = \ _d \texttt{M}^{\, \imath}_A$,  whenever $a_{ii} \gg 0$ for all $1\leq i \leq \frac{\l}2$.
\end{enumerate}
\end{prop}

\begin{proof}
All properties  follow readily from the constructions except the last one. 
As the hybrid monomial bases for $\Uidsl$ and $\Si$
are defined multiplicatively by the same procedure, we only need to show that Property~(4)  in the rank one case.
We remind that by construction the hybrid monomial basis in the rank one case is exactly the canonical basis. 
Hence Property~(4) at rank one is  exactly the statement of \cite[Proposition~A.21]{BKLW}. 
\end{proof}

Now since we have Proposition~\ref{Mi}, the commutative diagram \eqref{i-commutative}, Proposition~\ref{Ki-Mi} at hand,
the constructions and results in Sections~\ref{sec:coidealCB} in the $\jmath$-setting can be rerun for the $\imath$-setting. 
Let us outline them. 

\begin{prop}
\label{i-transfer}
Given $A\in \tilde \Xi^{\imath}$,  we have
\[
\xi^{\imath}_{-2} (\{ _{p\I} A\}) = \{ _{(p-2)\I} A\}, \quad
\wp_{\imath} (\{_{p\I} A\}) = \wp_{\imath} (\{ _{(p-2)\I} A\}),
\quad \mbox{for all even integer $p \gg 0$},
\]
where $_{p\I} A = A + p \I$.
\end{prop}

\begin{proof}
The same type arguments of the proof of Proposition \ref{prop:xiCB} work here.
\end{proof}

We define an equivalence on $\tilde \Xi^{\imath}$ by
$A \approx B$ if and only if  $A- B = p \mbb I$ for some even integer $p$. 
We set
$
\hat \Xi^{\imath} = \tilde \Xi^{\imath}/ \approx.
$
By Proposition ~\ref{i-transfer}, the following definition is well defined.

\begin{Def}
We define $b_{\check A} = \wp_{\imath} (\{_{p\mbb I} A\}) \in \dot{\U}^{\imath}(\mathfrak{sl}_{\l})$,  $\forall p\gg 0,  \check A \in \hat \Xi^{\imath}$.
\end{Def}

Now as we have the key properties established in Propositions~\ref{Ki-Mi}--\ref{i-transfer}, 
we are in a position to establish the $\imath$-counterparts of results on canonical bases in Sections~\ref{sec:CBcoidealj}--\ref{sec:CBcoidealj+},
whose similar proofs will be skipped.
Below is a summary of the $\imath$-counterparts of Theorem ~\ref{j-CB}, Theorem ~\ref{th:posCBj}, Proposition ~\ref{orthonormality}, 
and Theorem~\ref{thm:positiveform}.

\begin{thm}
\begin{enumerate}
\item
The set $\mbf B^{\imath}(\mathfrak{sl}_{\l}) = \{ b_{\check A} | \check A \in \hat \Xi^{\imath} \} $ 
forms a basis for $\dot{\U}^{\imath}(\mathfrak{sl}_{\l})$ and for $_{\A} \dot{\U}^{\imath}(\mathfrak{sl}_{\l})$. 

\item
The structure constants for the algebra $\dot{\U}^{\imath}(\mathfrak{sl}_{\l})$ with respect
to the basis $\mbf B^{\imath}(\mathfrak{sl}_{\l}) $ are positive (i.e., in $\mathbb N[v,v^{-1}]$).

\item
The form $\langle -, - \rangle_{\imath}$ on $\dot \U^{\imath}(\mathfrak{sl}_{\l}) $ is non-degenerate. 
Moreover, the basis $\mbf B^{\imath}(\mathfrak{sl}_{\l})$ is almost orthonormal and positive with respect to this  form, i.e., 
$\langle b_{\check A}, b_{\check B} \rangle_{\imath} \in \delta_{\check A, \check B} + v^{-1} \mathbb N [[v^{-1}]]$. 
\end{enumerate}
\end{thm}

Again, similar to Proposition ~\ref{3-property},
the signed canonical basis $-\mbf B^{\imath}(\mathfrak{sl}_{\l}) \cup \mbf B^{\imath}(\mathfrak{sl}_{\l})$  is characterized 
by the bar invariance, integrality and almost orthonormality.

\begin{rem}
\label{i-action}
The main results (Theorem~\ref{j-RLconj}, Proposition~\ref{prop:j-slsch}, Proposition~\ref{prop:j-sch}, Corollary~\ref{cor:positive}, 
Propositions~\ref{prop:sameCB}--\ref{prop:gljCB+})
in Sections~ \ref{pos-transfer-cb}--\ref{j-gl-cb}  admit $\imath$-analogues here with $n$ replaced by $\l$ and $\jmath$ by $\imath$, respectively. 
\end{rem}

\begin{rem}
Shigechi  \cite{Sh14} has established by combinatorial methods certain 
positivity of the $\imath$-canonical bases (introduced in \cite{BW13})  on general tensor products of modules
of the quantum coideal algebra of $\bold{U}(\sll_2)$, and this supports our general positivity conjectures. 
See Remark ~\ref{i-action} for a closely related result. 
\end{rem}


\section{Formulas of canonical basis of $\mbf{S}(2,d)$}
\label{rank-one}

\subsection{Combinatorial identities}

Recall the  quantum $v$-binomial coefficients were defined in \eqref{qnumber} for $m\in \Z$ and $b \in \mbb N$. 
We introduce the following additional notation
\[
\begin{bmatrix}
m\\
b
\end{bmatrix}_{v^2}
=\prod_{1\leq i\leq b} \frac{v^{4(m-i+1)}-1}{v^{4i}-1}.
\]

We first establish two combinatorial identities which are needed in later computations
and could be of some independent interest as well.

\begin{lem}
 \label{lem:sum=1}
For any $a \in \Z$ and $p \in \mathbb N$, we have 
\begin{align*}
 \sum_{s=0}^p  v^{2s(a+2s)}
\begin{bmatrix}
p\\s
\end{bmatrix}_{v^2}
 \prod_{k=1}^{p-s} (1-v^{2a+4s+4k}) =1. 
\end{align*}
\end{lem}

\begin{proof}
Recall the quantum binomial identity 
\begin{equation}
 \label{binomial id}
\begin{bmatrix}
p\\s
\end{bmatrix}_{v^2}
= \begin{bmatrix}
p-1\\s
\end{bmatrix}_{v^2} 
+
v^{4p-4s} \begin{bmatrix}
p-1\\s-1
\end{bmatrix}_{v^2}.
\end{equation}
We prove the lemma by induction on $p$, with the base case for $p=0$ being trivial.

By  \eqref{binomial id}, we can rewrite the sum as a sum of two summands:
\begin{align*}
 \sum_{s=0}^p  v^{2s(a+2s)}
\begin{bmatrix}
p\\s
\end{bmatrix}_{v^2}
 \prod_{k=1}^{p-s} (1-v^{2a+4s+4k}) =S_1 +S_2, 
 \end{align*}
where
 \begin{align*}
 S_1 &= 
 \sum_{s=1}^p  v^{2s(a+2s)} v^{4p-4s}
\begin{bmatrix}
p-1\\s-1
\end{bmatrix}_{v^2}
 \prod_{k=1}^{p-s} (1-v^{2a+4s+4k}),
 \\
 S_2 &= 
 \sum_{s=0}^{p-1}  v^{2s(a+2s)}
\begin{bmatrix}
p-1\\s
\end{bmatrix}_{v^2}
 \prod_{k=1}^{p-s} (1-v^{2a+4s+4k}).
\end{align*}

Setting $p'=p-1, s'=s-1$ and $a'=a+2$,   we have $a+2s =a'+2s'$, and thus by the inductive assumption (with $p'<p$) we obtain
 \begin{align*}
 S_1 &= 
 v^{2a+ 4p}
 \sum_{s'=0}^{p'}  v^{2s'(a'+2s')}  
\begin{bmatrix}
p' \\s'
\end{bmatrix}_{v^2}
 \prod_{k=1}^{p' -s'} (1-v^{2a'+4s'+4k})  =  v^{2a+ 4p}.
\end{align*}
Setting $p'=p-1$, by the inductive assumption (with $p'<p$) again we have
 \begin{align*}
 S_2 &= 
 \sum_{s=0}^{p'}  v^{2s(a+2s)}
\begin{bmatrix}
p' \\s
\end{bmatrix}_{v^2}
 \prod_{k=1}^{p'-s} (1-v^{2a+4s+4k}) \cdot (1-v^{2a+ 4p})
 =1-v^{2a+ 4p}.
\end{align*}
Summing up $S_1$ and $S_2$ above we have proved the lemma. 
\end{proof}

\begin{lem}
 \label{lem:sum=1b}
For $m\in \mathbb N$, we have
\begin{align*}
\sum_{j=0}^m   \frac{v^{(m-j)(m-j+1)} \prod_{u=1}^j (1 -v^{2(m-u+1)})}{\prod_{k=1}^{\lfloor \frac{j}2\rfloor}(1-v^{4k})} =1. 
\end{align*}
\end{lem}

\begin{proof}
Set $m =2n$ if $m$ is even or $m=2n+1$ otherwise.  We first sum up the two summands with $j=2d$ and $j=2d +1$, for fixed $d$ with $0\leq d \leq n$:
\begin{align*}
&v^{(m-2d)(m-2d+1)} \frac{\prod_{u=1}^{2d}  (1 -v^{2(m-u+1)})}{\prod_{k=1}^{d}(1-v^{4k})}
+ v^{(m-2d-1)(m-2d)} \frac{\prod_{u=1}^{2d+1}  (1 -v^{2(m-u+1)})}{\prod_{k=1}^{d}(1-v^{4k})}
\\
& =v^{(m-2d-1)(m-2d)} \frac{\prod_{u=1}^{2d}  (1 -v^{2(m-u+1)})}{\prod_{k=1}^{d}(1-v^{4k})}
\\
& =v^{(m-2d-1)(m-2d)}
\begin{bmatrix}
n \\d
\end{bmatrix}_{v^2}
 \prod_{k=1}^{d}  (1 -v^{4n-4d+4k \mp 2}),
\end{align*}
where the sign `$-$' is always taken for $m=2n$ and `$+$' for $m=2n+1$ on the right-hand side above and similar places below. 
Note that the above is actually valid for $d=n$ in case $m=2n$ as well, where the second summand on the left-hand side is simply zero. 

Hence, noting 
$
\begin{bmatrix}
n \\d
\end{bmatrix}_{v^2}
=
\begin{bmatrix}
n \\n-d
\end{bmatrix}_{v^2}$ and setting $s=n-d$, 
 we have
\begin{align*}
 \sum_{j=0}^m  & \frac{v^{(m-j)(m-j+1)} \prod_{u=1}^j (1 -v^{2(m-u+1)})}{\prod_{k=1}^{\lfloor \frac{j}2\rfloor}(1-v^{4k})} 
 \\
&= \sum_{d=0}^n v^{(2n-2d \mp 1)(2n-2d)} 
\begin{bmatrix}
n \\d
\end{bmatrix}_{v^2}
 \prod_{k=1}^{d}  (1 -v^{4n-4d+4k \mp 2})
 \\
&= \sum_{s=0}^n v^{2s(2s\mp 1)} 
\begin{bmatrix}
n \\s
\end{bmatrix}_{v^2}
 \prod_{k=1}^{n-s}  (1 -v^{4s+4k \mp2}) =1,
\end{align*}
where the last equation uses Lemma~\ref{lem:sum=1} (where we set $a =\mp 1$ and $p=n$). The lemma is proved. 
\end{proof}

\subsection{The bar conjugate of the standard basis}

Let \[
\mbb T = \bigsqcup_{d \ge 0} \mbf{S}^\imath (2,d)
\] 
be the $\Q(v)$-vector space with the standard basis $\{[A_{a,r}]\vert a, r \in \mathbb N\}$.
As before we set $[A_{a,r}]=0 \mbox{ if } a<0 \ \mbox{or}\ r <0.$
We introduce a shorthand notation to denote the monomial basis element $M_{a,r}={}_d\texttt{M}_{A_{a,r}}$. 
By \cite[(5.4)]{BKLW}, we have
\begin{align}
 \label{enfn}
M_{a,r}
= [A_{a,r}] + \sum^r_{i=1} v^{\beta_{a}(i)} \overline{\begin{bmatrix}
a + i\\
 i
\end{bmatrix}
} 
[A_{a+i, r-i}]
\end{align}
where 
\begin{equation}  \label{be}
\beta_{a}(i) 
=ai - \frac12 {i(i+1)}.
\end{equation}
Then $\{ M_{a,r} \vert a,r \in \mathbb N, a+b =d\}$ forms a monomial basis for $\mbf{S}^\imath(2,d)$, and so
$\{ M_{a,r} \vert a,r \in \mathbb N\}$ forms a monomial basis for $\mbb T$.
There is a $\Q$-linear bar involution on  $\mbf{S}^\imath(2,d)$ for all $d$ and hence on $\mbb T$, denoted by  $\overline{\phantom{r}}$, which fixes each $M_{a,r}$. 
Note that 
\begin{equation}
\label{binom}
\overline{\begin{bmatrix}
m\\
a
\end{bmatrix}
} =v^{2a(a-m)} \begin{bmatrix}
m\\
a
\end{bmatrix},
\quad \text{ and } \quad
\overline{M}_{a,r} =M_{a,r}. 
\end{equation}

The following theorem is obtained with help from a UVA undergraduate Tahseen Rabbani (supported by
NSF), whose
computer computation for small values of $r$ was crucial in formulating 
the precise statement.

\begin{thm} [joint with Tahseen Rabbani]
 \label{thm:Abar}
For all $a, r \in \mathbb N$, we have
\begin{align*}
\overline{[A_{a,r}]} &=\sum_{i=0}^r  
v^{-ia - {i+1 \choose 2}} \cdot
 \frac{\prod_{k=1}^i (1-v^{2(a+k)} )}{\prod_{k=1}^{\lfloor \frac{i}2\rfloor} (1-v^{4k})}  [A_{a+i, r-i}]
 \\
 &= \sum_{i=0}^r   
 \frac{\prod_{k=1}^i (v^{-a-k}-v^{a+k} )}{\prod_{k=1}^{\lfloor \frac{i}2\rfloor} (1-v^{4k})}  [A_{a+i, r-i}]. 
\end{align*}
\end{thm}
 
\begin{proof} 
The two expressions in the statement are clearly equal. 
We shall proceed by induction on $r$. The base case for $r=0$ is clear.

Assume the formula is verified for $\overline{[A_{a,r'}]}$ for all $a, r' \in \mathbb N$ such that $r'<r$.
By \eqref{enfn} and $\overline M_{a,r}  =M_{a,r}$, it suffices to verify the formula for $\overline{[A_{a,r}]}$ as given in the theorem satisfies that
\begin{align*}
\overline{[A_{a,r}]} + \sum^r_{i=1} v^{-\beta_{a}(i)} \begin{bmatrix}
a + i\\
 i
\end{bmatrix}
\overline {[A_{a+i, r-i}]}
= [A_{a,r}] + \sum^r_{i=1} v^{\beta_{a}(i)} \overline{\begin{bmatrix}
a + i\\
 i
\end{bmatrix}
} 
[A_{a+i, r-i}]
\end{align*}

Equating the coefficients of $[A_{a+m, r-m}]$ on both sides of the above identity, 
we are reduced to verifying the following identity 
for $0\le m \le r$:
\begin{align*}
\sum_{i+j=m} v^{-ai +\frac12 i(i+1)}
\begin{bmatrix}
a + i\\
 i
\end{bmatrix}
\frac{\prod_{k=1}^j (v^{-a-i-k} -v^{a+i+k})}{\prod_{k=1}^{\lfloor \frac{j}2\rfloor} (1 -v^{4k})}
=v^{-am -\frac12 m(m+1)} \begin{bmatrix}
a + m\\
 m
\end{bmatrix}.
\end{align*}
(We have used \eqref{binom} on deriving the right-hand side above.)

After further simplification using 
$\begin{bmatrix}
a + i\\
 i
\end{bmatrix} = \frac{[a+i]!}{[a]![i]!}$ and $i=m-j$, the above identity is reduced to the following identity for $m\ge 0$:
\begin{align*}
\sum_{j=0}^m v^{(m-j)(m-j+1)} \frac{[a+m-j]!}{[m-j]!} 
 \frac{\prod_{u=1}^j [a+m+1-u] \cdot (1 -v^2)^j}{\prod_{k=1}^{\lfloor \frac{j}2\rfloor}(1-v^{4k})} = \frac{[a+m]!}{[m]!}.
\end{align*}
Thanks to $[a+m-j]! \prod_{u=1}^j [a+m+1-u] =[a+m]!$, 
 the above identity is equivalent to the  identity   in Lemma~\ref{lem:sum=1b}. 
The theorem is proved. 
\end{proof}

Denote the coefficient of $A_{a+i, r-i}$ in  Theorem~\ref{thm:Abar} above,
which   is independent of $r$, by
\begin{equation}
 \label{eq:ba}
b^i_a =  v^{-ia - {i+1 \choose 2}} \cdot
  \frac{\prod_{k=1}^i (1-v^{2(a+k)} )}{\prod_{k=1}^{\lfloor \frac{i}2\rfloor} (1-v^{4k})}, 
\qquad  \text{with } b^0_a=1.
\end{equation}
Then we have
\begin{equation}
\label{eq:bAb}
\overline{[A_{a,r}]} =\sum_{i=0}^r   b_a^i \cdot  [A_{a+i, r-i}], \text{  for all } a, r \in \mathbb N.
\end{equation}

\begin{example}
For $a\in \mathbb N$, we have
\begin{align*}
\overline{[A_{a,0}]} &= [A_{a,0}], 
\qquad
\overline{[A_{a,1}]} = [A_{a,1}] +(v^{-a-1} -v^{a+1}) [A_{a+1,0}],
 \\
\overline{[A_{a,2}]} &= [A_{a,2}] +(v^{-a-1} -v^{a+1}) [A_{a+1,1}]
+ \frac{v^{-2a-3}(1 -v^{2(a+1)})( 1- v^{2(a+2)})}{1- v^4} [A_{a+2,0}].
\end{align*}
\end{example}

\subsection{Formulas for canonical basis of $\mbf{S}^\imath(2,d)$}

\label{rank-one-iCB}

The canonical basis is the $\Q(v)$-basis $\{ \{A_{a,r}\} \vert a,r \in \mathbb N\}$  for $\mbb T$, which is completely determined by the 
bar invariance together with the following property:
\begin{align}
\label{eq:iB}
\{A_{a,r}\}
= [A_{a,r}] + \sum^r_{i=1} \gamma_{a,r}(i) [A_{a+i, r-i}], \qquad \text{ for } \gamma_{a,r}(i) \in v^{-1} \Z[v^{-1}].
\end{align}
We denote $\gamma_{a,r}(0) =1$. 
\begin{lem}
The polynomials $\gamma_{a,r}(i)$ are independent of $r$; we shall write $\gamma_{a}(i) =\gamma_{a,r}(i)$.
\end{lem}

\begin{proof}
We shall show by induction on $i\ge 0$.  
The case for $i=0$ is clear. 

By \eqref{eq:iB}, we have
$ \sum^r_{i=0} \gamma_{a,r}(i) [A_{a+i, r-i}] =  \sum^r_{j=0} \overline \gamma_{a,r}(j) \overline{[A_{a+j, r-j}]}.
$
Equating the coefficients of $[A_{a+i, r-i}]$ on both sides of this equation with the help of \eqref{eq:bAb} gives us
\begin{equation} 
\label{eq:Gbar}
\gamma_{a,r}(i) -\overline \gamma_{a,r}(i) = \sum_{j=0}^{i-1} \overline \gamma_{a,r}(j) \, b_{a+j}^{i-j}. 
\end{equation}
It follows from this and an easy induction on $i$ that $\gamma_{a,r}(i)$ is independent of $r$. 
\end{proof}
The next theorem establishes formulas for the canonical basis $\{A_{a,r}\}$ for $\mbf{S}^\imath(2,d)$ for all $d$, 
or equivalently by \eqref{eq:iB}, determines $\gamma_{a}(i)$ for all $a, i\in \mathbb N$. 

\begin{thm}
 \label{th:iCB1}
\begin{enumerate}
\item For $a, s \in \mathbb N$ with $a$ even, we have 
\begin{align*}
\gamma_a(2s) 
 & = v^{-2s^2-s}   \prod_{k=1}^s \frac{1-v^{-2a-4k}}{1-v^{-4k}},  
 \\
\gamma_a(2s+1) 
&=  v^{-a-2s^2-3s-1}   \prod_{k=1}^s \frac{1-v^{-2a-4k}}{1-v^{-4k}}.
\end{align*}

\item For $a, s \in \mathbb N$ with $a$  odd, we have 
\begin{align*}
\gamma_a(2s) 
&=  v^{-2s^2+s}  \prod_{k=1}^s \frac{1-v^{-2a-4k-2}}{1-v^{-4k}},   
 \\
\gamma_a(2s+1) 
&= v^{-a-2s^2-s-1}  \prod_{k=1}^s \frac{1-v^{-2a-4k-2}}{1-v^{-4k}}.
\end{align*}
\end{enumerate}
\end{thm}

In other words, these polynomials $\gamma_a(r)$ are all essentially $v^{2}$-binomial coefficients.
 
\begin{proof} 
Let us rewrite \eqref{eq:Gbar} as
\begin{equation}
 \label{eq:ga}
\gamma_{a}(r)  = \sum_{i=0}^{r} \overline \gamma_{a}(i) \; b_{a+i}^{r-i}. 
\end{equation}
This formula uniquely determines the polynomials $\gamma_a(r)$ for all $a, r \in \mathbb N$  (by induction on $r$), 
which satisfy $\gamma_a(0)=1$ and $\gamma_a(r) \in v^{-1} \Z[v^{-1}]$ for $r\ge 1$. 
It suffices to verify that the formulas for $\gamma_a(r)$ given in the theorem do satisfy \eqref{eq:ga}.
The verification is divided into 4 very similar cases, depending on the parity of $a$ and the parity of $r$. 

Assume first that  both $a$ and $r$ are odd. Set $r=2p+1$. 
Let $0\le s \le p$.
We have
\begin{align*}
 \overline \gamma_{a}(2s) \, b_{a+2s}^{r-2s}
& = v^{2s(a+2s) -ar - {r+1 \choose 2}}
\prod_{k=1}^s \frac{1-v^{2a+4k+2}}{1-v^{4k}} \cdot
  \frac{\prod_{u=1}^{r-2s} (1-v^{2a+4s+2u} )}{\prod_{k=1}^{\lfloor \frac{r}2\rfloor-s} (1-v^{4k})},
  \\
  \overline \gamma_{a}(2s+1) \, b_{a+2s+1}^{r-2s-1}
& = v^{2s(a+2s) -ar - {r+1 \choose 2} + 2a+4s+2}
\prod_{k=1}^s \frac{1-v^{2a+4k+2}}{1-v^{4k}} \cdot
  \frac{\prod_{u=1}^{r-2s-1} (1-v^{2a+4s+2u+2} )}{\prod_{k=1}^{\lfloor \frac{r-1}2\rfloor -s} (1-v^{4k})}. 
\end{align*}
The above two formulas have almost identical factors  except  that $ \overline \gamma_{a}(2s)  b_{a+2s}^{r-2s}$ has 
an extra factor $(1-v^{2a+4s+2})$ while $ \overline \gamma_{a}(2s+1)  b_{a+2s+1}^{r-2s-1}$ has an extra factor $v^{2a+4s+2}$.
Hence,
\begin{align*}
\sum_{i=0}^{r} \overline \gamma_{a}(i) \, b_{a+i}^{r-i}
&= 
\sum_{s=0}^p \Big(\overline \gamma_{a}(2s)\, b_{a+2s}^{r-2s}
+
  \overline \gamma_{a}(2s+1)\, b_{a+2s+1}^{r-2s-1} \Big)
  \\
&= \sum_{s=0}^p v^{2s(a+2s) -ar - {r+1 \choose 2}}
  \frac{\prod_{k=1}^s (1-v^{2a+4k+2})}{\prod_{k=1}^s(1-v^{4k})} \cdot
 \frac{\prod_{u=1}^{2p-2s} (1-v^{2a+4s+2u+2} )}{\prod_{k=1}^{p-s} (1-v^{4k})}.
\end{align*}
Using
$$\prod_{u=1}^{2p-2s} (1-v^{2a+4s+2u+2} ) = \prod_{k=s+1}^p (1-v^{2a+4k+2}) \cdot \prod_{k=1}^{p-s} (1-v^{2a+4s+4k}),
$$
we can rewrite the above equation as
\begin{align}
 \label{eq:rhs}
\sum_{i=0}^{r} \overline \gamma_{a}(i) \; b_{a+i}^{r-i}
&= v^{-ar - {r+1 \choose 2}}  \sum_{s=0}^p  v^{2s(a+2s)}
  \frac{\prod_{k=1}^p (1-v^{2a+4k+2})}{\prod_{k=1}^s(1-v^{4k}) \prod_{k=1}^{p-s} (1-v^{4k})} \cdot
 \prod_{k=1}^{p-s} (1-v^{2a+4s+4k}).
\end{align}
On the other hand, we have
\begin{equation}
 \label{eq:lhs}
\gamma_{a}(r) =v^{-ar - {r+1 \choose 2}}  
  \frac{\prod_{k=1}^p (1-v^{2a+4k+2})}{\prod_{k=1}^p (1-v^{4k})}.
\end{equation}
Therefore, the verification of the identity \eqref{eq:ga} follows from \eqref{eq:rhs}-\eqref{eq:lhs} and the identity in Lemma~\ref{lem:sum=1}, and the theorem is proved in the case
when both $a$ and $r$ are odd.

In the remaining three cases when not both $a$ and $r$ are odd, we have  analogous reductions of verification of \eqref{eq:ga}
to the same identity in Lemma~\ref{lem:sum=1}, and we shall skip the details.  
\end{proof}

\begin{example}
We have $\gamma_a(0)  =1, 
\gamma_a(1)  = v^{-a-1}$, and
\begin{align*}
\gamma_a(2)  =
 \begin{cases}  
\frac{v^{-3} -v^{-2a-7}}{1-v^{-4}},  & \text{ for } a \text{ even}  \\
\frac{v^{-1} -v^{-2a-7}}{1-v^{-4}}, & \text{ for } a \text{ odd}.
\end{cases} 
\end{align*}
Then we have
\begin{align*}
\{A_{a,0}\} = [A_{a,0}],
\quad \{A_{a,1}\} &= [A_{a,1}] +v^{-a-1} [A_{a+1,0}],
\\
\{A_{a,2}\} &=[A_{a,2}] +v^{-a-1} [A_{a+1,1}] +  \gamma_a(2) [A_{a+2,0}].
\end{align*}
\end{example}



\end{document}